\documentclass[10pt]{amsart}
\usepackage{amsrefs}
\usepackage{amssymb}
\usepackage[all]{xy}

\usepackage{hyperref}
\usepackage{systeme}
\usepackage{array}

\usepackage{tikz}
\usetikzlibrary{matrix,arrows}

\DeclareMathOperator{\Img}{Im}

\DeclareMathOperator{\Ker}{Ker}

\DeclareMathOperator{\res}{res}

\DeclareMathOperator{\Zen}{Z}

\DeclareFontFamily{U}{wncy}{}
\DeclareFontShape{U}{wncy}{m}{n}{<->wncyr10}{}
\DeclareSymbolFont{mcy}{U}{wncy}{m}{n}
\DeclareMathSymbol{\Sha}{\mathord}{mcy}{"58}
\DeclareMathSymbol{\sha}{\mathord}{mcy}{"78}

\begin{document}

\newtheorem{thm}{Theorem}[section]
\newtheorem{cor}[thm]{Corollary}
\newtheorem{lem}[thm]{Lemma}
\newtheorem{fact}[thm]{Fact}
\newtheorem{prop}[thm]{Proposition}
\newtheorem{defin}[thm]{Definition}
\newtheorem{exam}[thm]{Example}
\newtheorem{examples}[thm]{Examples}
\newtheorem{rem}[thm]{Remark}
\newtheorem{case}{Case}
\newtheorem{claim}{Claim}
\newtheorem{question}[thm]{Question}
\newtheorem{conj}[thm]{Conjecture}
\newtheorem*{notation}{Notation}
\swapnumbers
\newtheorem{rems}[thm]{Remarks}
\newtheorem*{acknowledgment}{Acknowledgments}
\newtheorem*{thmno}{Theorem}

\newtheorem{questions}[thm]{Questions}
\numberwithin{equation}{section}

\newcommand{\gr}{\mathrm{gr}}
\newcommand{\inv}{^{-1}}
\newcommand{\isom}{\cong}
\newcommand{\dbC}{\mathbb{C}}
\newcommand{\F}{\mathbb{F}}
\newcommand{\dbN}{\mathbb{N}}
\newcommand{\Q}{\mathbb{Q}}
\newcommand{\dbR}{\mathbb{R}}
\newcommand{\dbU}{\mathbb{U}}
\newcommand{\Z}{\mathbb{Z}}
\newcommand{\calG}{\mathcal{G}}
\newcommand{\K}{\mathbb{K}}
\newcommand{\rmH}{\mathrm{H}}
\newcommand{\bfH}{\mathbf{H}}
\newcommand{\bfA}{\mathbf{A}}
\newcommand{\rmr}{\mathrm{r}}
\newcommand{\Span}{\mathrm{Span}}
\newcommand{\eue}{\mathbf{e}}
\newcommand{\bfLam}{\mathbf{\Lambda}}
\newcommand{\calX}{\mathcal{X}}
\newcommand{\calY}{\mathcal{Y}}
\newcommand{\calV}{\mathcal{V}}
\newcommand{\calE}{\mathcal{E}}
\newcommand{\calW}{\mathcal{W}}

%%%%%%%%%%%%%%%%%
% New commands for our things - such as cocycles and so on

\newcommand{\hac}{\hat c}
\newcommand{\hatheta}{\hat\theta}
%%%%%%%%%%%%%%%%%%%%%%%%%%%%%%%%%%

\title[Digraphs, pro-$p$ groups and Massey products]{Digraphs, pro-$p$ groups \\ and Massey products in Galois cohomology}

\author{Claudio Quadrelli}
\address{Department of Science \& High-Tech, University of Insubria, Como, Italy EU}
\email{claudio.quadrelli@uninsubria.it}
\date{\today}
%\thanks{W la pastasciutta!}
%\dedicatory{To Pablo Spiga, an enthusiast algebraist \\ and a ``graphomaniac'', with admiration.}

\begin{abstract}
Let $p$ be a prime. We characterize the oriented right-angled Artin pro-$p$ groups whose
$\F_p$-cohomology algebra yields no essential $n$-fold Massey products for every $n>2$, in terms of the associated digraph.
Moreover, we show that the $\F_p$-cohomology algebra of such an oriented right-angled Artin pro-$p$ group is isomorphic to the exterior Stanley-Reisner $\F_p$-algebra associated to the same digraph.
\end{abstract}



\maketitle
%%%%%%%%%%%%%%%%%%%%%%%%%%%%%%%%%%%%%%%%%%%%%%%%%

\section{Introduction}
\label{sec:intro}

\subsection{Framework}\label{ssec:frame}
Let $X$ be a complex and $R$ a commutative ring.
The $R$-cohomology groups $\rmH^i(X,R)$, $i\geq0$, are equipped with the cup-product
\[
 \smallsmile \colon \rmH^s(X,R)\otimes\rmH^t(X,R)\longrightarrow\rmH^{s+t}(X,R)
\]
induced by the product of $R$, which turn the space $\coprod_i\rmH^i(X,R)$ into a ring.
Massey products are multi-valued higher order cohomology operations of several variables which generalize the cup-product.
More in detail, if $\alpha_1,\ldots,\alpha_n$ is a sequence of length $n$ of (non-necessarily distinct) elements of $\rmH^1(X,R)$, the ``value'' of the $n$-fold Massey product associated to the above sequence is a subset
\[
 \langle\alpha_1,\ldots,\alpha_n\rangle\subseteq\rmH^2(X,R),
\]
which may be empty.
If $n=2$, then the 2-fold Massey product $\langle\alpha_1,\alpha_2\rangle$ is the subset of $\rmH^2(X,R)$ containing only $\alpha_1\smallsmile\alpha_2$.
If an $n$-fold Massey product $\langle\alpha_1,\ldots,\alpha_n\rangle$ is not empty and it does not contain 0, then it is said to be essential.
(For an overview on Massey products, accessible to non-experts in cohomology, see, e.g., \cite{wang:Massey}.)

The presence of essential Massey products in cohomology reveals information which cannot be revealed by the ring structure structure.
One of the most famous examples of this phenomenon is given by the { Borromean rings}.
\begin{center}
 \includegraphics[scale=0.1]{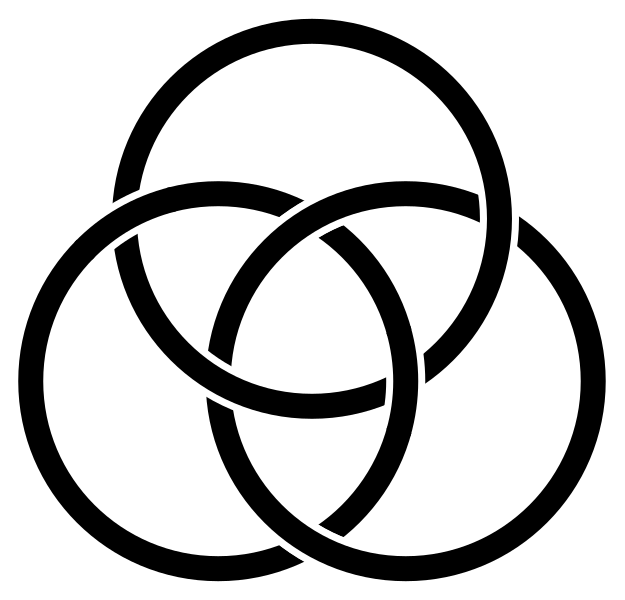}
\end{center}
If one considers singular cohomology of the complement of the Borromean rings, the ring structure tells just that the rings are pairwise disjoint, while the existence of essential 3-fold Massey products explains that the Borromean rings are not equivalent to three unconnected circles\footnote{Even if the Borromean rings showed up in several contexts in past cultures, they take the name from the Italian family { Borromeo}, as the rings appear on the coat of arms of the family. During the Reinassance, the Borromeo were the bankers of the dukes of Milan --- still, the most important memeber of the family was St.~Charles ---, the families Visconti and Sforza: the three rings are told to represent the fortunes of the three families, which are tightly linked to each other.} (see, e.g., \cite[\S~2.2]{wang:Massey}).

%For example, if one considers singular cohomology of the complement of the Borromean rings, the ring structure tells just that the rings are pairwise disjoint, while the existence of essential 3-fold Massey products explains that the Borromean rings are not equivalent to three unconnected circles (see, e.g., \cite[\S~2.2]{wang:Massey}).

One of the topics which gained great interest in Galois theory, in recent years, is the
study of Massey products in Galois cohomology (for a more detailed overview on Massey products in Galois cohomology we direct the reader to \cite{ido:Massey} and \cite{mt:Massey}).
Given a prime number $p$, let $\K$ be a field containing a root of 1 of order $p$.
Also, let $G_{\K}(p)$ denote the maximal pro-$p$ Galois group of $\K$, namely, $G_{\K}(p)$ is the Galois group of the compositum of all Galois $p$-extensions of $\K$ --- equivalently, $G_{\K}(p)$ is the maximal pro-$p$ quotient of the absolute Galois group of $\K$.
Consider the field with $p$-elements $\F_p$ as a trivial $G_{\K}(p)$-module, and the induced $\F_p$-cohomology algebra
$$\bfH^\bullet(G_{\K}(p),\F_p)=\coprod_{i\geq0}\rmH^i(G_{\K}(p),\F_p).$$
The recent hectic research on Massey products in Galois cohomology started after the work \cite{hopwick} of M.J.~Hopkins and K.G.~Wickelgren, where they proved that if $\K$ is a global field of characteristic
not 2, then in $\bfH^\bullet(G_{\K}(2),\F_2)$ there are no essential 3-fold Massey products,
and moreover they conjectured that this is true for any field of characteristic not 2.
Shortly after, in \cite{MT:kernel} J.~Mina\v{c} and N.D.~T\^an conjectured that for every prime $p$ and for any field $\K$ containing a root of 1 of order $p$, in $\bfH^\bullet(G_{\K}(p),\F_p)$ there are no essential $n$-fold Massey products for every $n>2$.
The main results obtained in this direction are the following:
\begin{itemize}
 \item[(a)] E.~Matzri proved that Mina\v c-T\^an's conjecture holds true for $n=3$ (see the preprint \cite{eli:Massey}, see also the published works \cite{EM:Massey,MT:Masseyall});
\item[(b)] J.~Mina\v c and N.D.~T\^an proved their own conjecture for local fields (see \cite{mt:Massey});
\item[(c)] Y.~Harpaz and O.~Wittenberg proved Mina\v c-T\^an's conjecture for number fields (see \cite{HW:Massey});
\item[(d)] the author proved a strengthened version of Mina\v c-T\^an's conjecture for fields satisfying I.~Efrat's Elementary Type Conjecture (see \cite{cq:massey,cq:pyt});
\item[(e)] A.~Merkurjev and F.~Scavia proved that Mina\v c-T\^an's conjecture holds true also for $n=4$ (see \cite{MerSca1}).
\end{itemize}
The absence of essential 3-fold Massey products in the $\F_p$-cohomology of maximal pro-$p$ Galois groups provided new obstructions for the realization of pro-$p$ groups as absolute Galois groups of fields (see \cite[\S~7]{mt:Massey}), which was a very remarkable achievement.

Moreover, there are number-theoretic analogues of the Borromean rings, giving essential Massey products in Galois cohomology (see, e.g., \cite{jochen,morishita,vogel}).
Further interesting results on Massey products in Galois cohomology have been obtained
by various authors (see, e.g., \cite{GPM,PJ,LLSWW,eli2,eli3,MerSca3,wick}).

It is therefore of major interest, in current research in Galois theory, to study Massey products in the $\F_p$-cohomology of pro-$p$ groups.

\subsection{Digraphs and oriented pro-$p$ RAAGs}
Within this frame, we focus on Massey products in the $\F_p$-cohomology of {oriented right-angled Artin pro-$p$ groups} associated to {digraphs}.
By an {digraph} $\Gamma$ we mean a pair of finite sets $\Gamma=(\calV,\calE)$ --- we tacitly assume that $\calV\cap\calE=\varnothing$ --- where $\calV$ is said to be the set of vertices of $\Gamma$, and $\calE$ is the set of directed edges of $\Gamma$, which consists of ordered couples of distinct vertices ---
i.e., $$\calE\subseteq \calV\times\calV\smallsetminus\{(v,v)\:\mid\:v\in\calV\}.$$
In other words, all digraphs in this paper have no loops nor parallel directed edges --- see, e.g., \cite[\S~1.10]{graph:book}.
As an example, the diagram
\begin{equation}\label{eq:exagraph intro}
 \xymatrix@R=1.5pt{&v_1& \\&\bullet& \\  \\  \\
 \bullet\ar[uuur] &\circ\ar[l]\ar@/_/[r]\ar[uuu]& \circ\ar@/_/[uuul]\ar@/_/[l] \\ v_2&v_3&v_4}
 \end{equation}
is the geometric representation of a digraph $\Gamma=(\calV,\calE)$ with four vertices: the black ones are those vertices which are terminal vertices of directed edges whose inverse does not belong to $\calE$.

A digraph with only undirected edges may be considered as an ``undirected'' graph, where the couples of edges $(v,w),(w,v)$ are identified with the subset $\{v,w\}\subseteq\calV$ (see Remark~\ref{rem:non or graph} below).

Now for a prime $p$ put $q=p^f$ for some $f\in\dbN\smallsetminus \{0\}$ --- we require also $f\geq2$ in the case $p=2$.
The {oriented right-angled Artin pro-$p$ group} (oriented pro-$p$ RAAG for short) associated to a digraph $\Gamma=(\calV,\calE)$ and to $q$ is the pro-$p$ group with pro-$p$ presentation
\[
 G=\left\langle\:v\in\calV\:\mid\: wuw^{-1}=\begin{cases}
                                             u^{1+q}& \text{if }(w,u)\text{is special}, \\
                                             u& \text{if }(w,u)\text{is ordinary},
        \end{cases}\;\forall\; (u,w)\in\calE \:\right\rangle.
\]
Observe that if a digraph $\Gamma$ has only undirected edges, then $G$ is the pro-$p$ RAAG (namely, the pro-$p$ completion of the discrete RAAG) associated to $\Gamma$ seen a graph (and it does not depend on $q$).

Oriented pro-$p$ RAAGs associated to digraphs have been studied in \cite{BQW}, where a digraph is called an ``oriented graph'' (though, not every digraph is an oriented graph as defined in \cite[\S~1.10]{graph:book}).
The family of oriented pro-$p$ RAAGs associated to digraphs is extremely rich, and it includes free pro-$p$ groups, free abelian pro-$p$ groups, certain families of $p$-adic analytic pro-$p$ groups and even some finite $p$-groups (see, e.g., \cite[\S~1]{BQW}).
For this reason, pro-$p$ groups associated to digraphs have been object of study in recent times, especially from a Galois-theoretic perspective (see, e.g., \cite{BCQ,BQW,CasQua,QSV,SZ}).

\subsection{Main results}
Our main goal is to characterize oriented pro-$p$ RAAGs whose $\F_p$-cohomology algebra yields no essential Massey products in terms of the associated digraph.

\begin{thm}\label{thm:main intro}
Let $\Gamma=(\calV,\calE)$ be a digraph, and let $G$ be the oriented pro-$p$ RAAG associated to $\Gamma$ and to a $p$-power $p^f$ {\rm (}with $f\geq2$ in case $p=2${\rm )}.
Then the following are equivalent:
\begin{itemize}
 \item[(i)] for every $n>2$ there are no essential $n$-fold Massey products in $\bfH^\bullet(G,\F_p)$;
 \item[(ii)] $G$ satisfies the strong $n$-Massey vanishing property, with respect to $\F_p$ --- namely, if $\alpha_1,\ldots,\alpha_n$ is a sequence of elements of $\rmH^1(G,\F_p)$ satisfying
\[ \alpha_1\smallsmile\alpha_2=\alpha_2\smallsmile\alpha_3=\ldots=\alpha_{n-1}\smallsmile\alpha_n=0,\]
then the associated $n$-fold Massey product contains $0$ --- for every $n>2$;
 \item[(iii)] the digraph $\Gamma$ is special-clique.
\end{itemize}
\end{thm}

Here a digraph $\Gamma=(\calV,\calE)$ is said to be {special-clique} if it satisfies the following two conditions:
\begin{itemize}
 \item[(a)] every vertex which is the terminal vertex of a directed edge (i.e., which is black) is a sinkhole;
 \item[(b)] the initial vertices of directed edges with the same terminal vertex are all joined to each other.
\end{itemize}
If $\Gamma$ satisfies (at least) condition (a), then it is said to be a {special digraph}.
(For the detailed definitions see Definition~\ref{defin:special graph} below).
For example, the digraph represented in \eqref{eq:exagraph intro} is not special-clique, nor special: both conditions (a)--(b) are not satisfied, as the vertex $v_1$ is a sinkhole, but $v_2$ is not --- as for condition~(a) ---, and the vertices $v_2,v_4$ are origins of directed edges pointing both at the sinkhole $v_1$ but are not joined to each other --- as for condition~(b).

We prove Theorem~\ref{thm:main intro} using the ``pro-$p$ translation'' of a result of W.~Dwyer (see \cite{dwyer}) which interprets the existence of Massey products in the $\F_p$-cohomology of a group $G$ in terms of the existence of certain representations from $G$ to the group of upper unitriangular matrices with entries in $\F_p$ (see Proposition~\ref{prop:unipotent representation} below).

Also, we exploit the fact that the $\F_p$-cohomology groups of degree 1 and 2 of an oriented pro-$p$ RAAG are completely described in terms of the incidence structure of the associated digraph: namely, if $G$ is the oriented pro-$p$ RAAG associated to a digraph $\Gamma=(\calV,\calE)$ and to a $p$-power $q$, then
$\rmH^1(G,\F_p)\simeq\Lambda_1(\Gamma^\ast)$ and $\rmH^2(G,\F_p)\simeq\Lambda_2(\Gamma^\ast)$,
where $\Lambda_i(\Gamma^\ast)$ is the subspace of degree $i$ of the {exterior Stanley-Reisner $\F_p$-algebra} $\bfLam_\bullet(\Gamma^\ast)$ associated to $\Gamma$.

It is worth underlining that Theorem~\ref{thm:main intro} presents a phenomenon which is analogous to the example of the Borromean rings: Massey products in the $\F_p$-cohomology of oriented pro-$p$ RAAGs detect a combinatorial property of the underlying graphs --- being special-clique ---, while the $\F_p$-cohomology groups of degree 1 and 2 (which depend only on the incidence structure of the underlying graphs, as stated above) do not.

It is well-known that the cohomology algebra with coefficients in an arbitrary field of the discrete RAAG associated to a graph (with only undirected edges) is the associated exterior Stanley-Reisner algebra over that field (cf., e.g., \cite[Thm.~1.2]{bartholdi} or \cite[\S~3.2]{papa:raags}).
Moreover, by a result of K.~Lorensen, one knows that the $\F_p$-cohomology of the pro-$p$ RAAG associated to a graph (with only undirected edges) is isomorphic to the associated exterior Stanley-Reisner $\F_p$-algebra (see \cite[Thm.~2.6]{lorensen}).
On the other hand, it is an open problem to determine which digraphs yield oriented pro-$p$ RAAGs whose whole $\F_p$-cohomology algebra is isomorphic to the whole associated exterior Stanley-Reisner $\F_p$-algebra (see \cite[\S~5.6]{QSV}).
We prove that this happens for oriented pro-$p$ RAAGs whose $\F_p$-cohomology has no essential Massey products.

\begin{thm}\label{thm:quad intro}
Let $\Gamma=(\calV,\calE)$ be a special-clique digraph, and let $G$ the oriented pro-$p$ RAAG associated to $\Gamma$ and to a $p$-power $q=p^f$ {\rm(}with $f\geq2$ in case $p=2${\rm)}.
Then $$\bfH^\bullet(G,\F_p)\simeq\bfLam_\bullet(\Gamma^\ast).$$
\end{thm}

We prove Theorem~\ref{thm:quad intro} following the strategy outlined in \cite[\S~5.6]{QSV},
and employing the combinatorial properties of special-clique digraphs.

Theorems~\ref{thm:main intro}--\ref{thm:quad intro} provide concrete examples of pro-$p$ groups whose $\F_p$-cohomology: gives rise to no essential Massey products --- and, complementarily, with essential Massey products ---; is a quadratic $\F_p$-algebra (see \S~\ref{ssec:stanreis} below).
As it often happens in profinite group theory, it is very important to find concrete examples of profinite groups satisfying certain given properties, especially when there there is an astounding lack of such examples (as in the case of pro-$p$ groups with quadratic $\F_p$-cohomology, see \cite[\S~1.1]{QSV}).
At this aim, it is worth stressing that the examples of pro-$p$ groups, whose $\F_p$-cohomology gives rise to essential Massey products, studied so far have presentations whose relations involve higher commutators (see, e.g., \cite[\S~7]{mt:Massey}); while oriented pro-$p$ RAAGs have relations involving only elementary commutators times, possibly, powers of generators.

One of the aims of the present work is to provide an introduction to the study of Massey products in the $\F_p$-cohomology of pro-$p$ groups --- with a mostly group-theoretic perspective and quite a concrete approach --- for non-specialists in Galois cohomology, accessible (and, hopefully, appealing) to a broad audience, in particular to graduate students working in profinite group theory.
Moreover, we believe that the techniques used to prove Theorem~\ref{thm:main intro} may be useful for further future investigations on Massey products in Galois cohomology.

\subsection{Structure of the paper}
In \S~\ref{sec:graphs} we recall some definitions on digraphs (cf. \S~\ref{ssec:graphs}), and we study some properties of special-clique graphs (cf. \S~\ref{ssec:special-clique}--\ref{ssec:patching}). Also, we recall the definition of the exterior Stanley-Reisner algebra associated to a digraph (cf. \S~\ref{ssec:stanreis}).

In \S~\ref{sec:RAAGs}, after introducing some notation (cf. \S~\ref{ssec:not}), we recall some properties of oriented pro-$p$ RAAGs, providing several concrete examples (cf. \S~\ref{ssec:or}--\ref{ssec:or prop}).

In \S~\ref{sec:H} we provide a short {vademecum} on $\F_p$-cohomology of pro-$p$ groups (cf. \S~\ref{ssec:cohom}), for the convenience of the reader, and we recall some facts on the $\F_p$-cohomology of oriented pro-$p$ RAAGs (cf. \S~\ref{ssec:cohom RAAGS}).
Then, we prove Theorem~\ref{thm:quad intro} (cf. \S~\ref{ssec:H specialclique}).

In \S~\ref{sec:massey} we recall some definitions on Massey products, and the ``translation'' of Massey products in terms of upper unitriangular representations. Also, we prove some facts on upper unitriangular matrices with entries in $\F_p$ which will be used to prove Theorem~\ref{thm:main intro}.

Finally, \S~\ref{sec:massey RAAGs} is devoted to the proof of Theorem~\ref{thm:main intro}: the equivalence between statement~(i) and statement~(iii) is proved in \S~\ref{ssec:i}, while the equivalence between statement~(ii) and statement~(iii) is proved in \S~\ref{ssec:ii}.

 %%%%%%%%%%%%%%%%%%%%%%%%%%%%%%%%%%%%%%%%%%%%

{\small \subsection*{Acknowledgments}
The author wishes to thank: I.~Efrat, E.~Matzri, J.~Mina\v{c}, F.W.~Pasini, and N.D.~T\^an, for several inspiring discussions on Massey products in Galois cohomology; and S.~Blumer, A.~Cassella, I.~Foniqi, I.~Snopce, M.~Vannacci, and Th.S.~Weigel, for several inspiring discussions on digraphs and pro-$p$ groups.

%Last, but not least, the author is very grateful to the referees for the careful work carried with the manuscript and for the extremely useful comments and suggestions.%, and for inviting the author to consider also the case of fields not containing $\sqrt{-1}$ and with maximal pro-2 Galois group of elementary type.
\noindent The author acknowledges his membership to the national group GNSAGA (Algebraic Structures and Algebraic Geometry) of the National Institute of Advanced Mathematics -- a.k.a. INdAM -- ``F. Severi''.
}

%%%%%%%%%%%%%%%%%%%%%%%%%%%%%%%%%%%%%%%%%%%%%%%%%%%%%%%%%%%%%%%%%%%%%%%%5
%%%%%%%%%
%%%%%%%%%%%%%%%%%%%%%%%%%%%%%%%%%%%%%%%%%%%%%%%%%%%%%%%%%%%%%%%%%%%%%%%%

\section{Digraphs and pro-$p$ groups}
\label{sec:graphs}

%%%%%%%%%%%%%%%%%%%%%%%%%%%%%%%55

\subsection{Digraphs}\label{ssec:graphs}
The formal definition of a digraph --- as it may be found, e.g., in \cite[\S~1.10]{graph:book} --- is the following: a pair of disjoint sets $\Gamma=(\calV,\calE)$ together with two maps $i\colon \calE\to\calV$ and $t\colon \calE\to\calV$ --- the first one gives the initial vertex of a directed edge, and the second one gives the terminal vertex of a directed edge
Since we assume throughout the paper that every digraph has no loops --- i.e., no directed edges $\mathbf{e}\in\calE$ satisfying $i(\mathbf{e})=t(\mathbf{e})$ ---, nor parallel directed edges --- i.e., two directed edges $\mathbf{e}_1,\mathbf{e}_2\in\calE$ are parallel if $i(\mathbf{e}_1)=i(\mathbf{e}_2)$ and $t(\mathbf{e}_1)=t(\mathbf{e}_2)$ ---, we can consider $\calE$ as a subset of $\calV\times\calV\smallsetminus\{(v,v)\}$, as done in the Introduction.

In the following we recall some definitions which generalize some well-known notions on graphs (cf., e.g., \cite[pp.~3--4]{graph:book}).

\begin{defin}\rm
Let $\Gamma=(\calV,\calE)$ be a digraph.
 \begin{itemize}
 \item[(a)] $\Gamma$ is said to be {complete} if for any couple $v,w\in\calV$, one has $(v,w)\in\calE$ or $(w,v)\in\calE$.
  \item[(b)] An {induced subdigraph} of $\Gamma$ is a digraph $\Gamma'=(\calV',\calE')$ such that $\calV'\subseteq\calV$ and
  \[\calE'=\calE\cap(\calV'\times\calV');\]
  it is said to be {proper} if $\calV'\subsetneq\calV$ and $\calV'\neq\varnothing$.
  In particular, an induced subdigraph of $\Gamma$ which is a complete digraph is called a {clique} of $\Gamma$.
  \item[(c)] The {star} of a vertex $v\in\calV$ is the induced subdigraph $\mathrm{St}(v)=(\calV_v,\calE_v)$ of $\Gamma$ whose vertices are $v$ and all other vertices of $\Gamma$ which are adjacent to $v$, i.e.,
  \[
   \calV_v=\{\:v\:\}\cup\{\:w\:\mid\:(v,w)\in\calE\text{or }(w,v)\in\calE\:\}.
  \] \end{itemize}
\end{defin}

For digraphs, we will make use of the following.

\begin{defin}\label{defin:for digraph}\rm
Let $\Gamma=(\calV,\calE)$ be a digraph.
\begin{itemize}
 \item[(a)] We call a vertex $w\in\calV$ a special vertex if there exists at least another vertex $v\in\calV$ such that $(v,w)\in\calE$ but $(w,v)\notin\calE$; otherwise we call $v$ an ordinary vertex.
 \item[(b)] We call a special vertex $w$ a sinkhole if for any other vertex $v\in\calV$ such that $(v,w)\in\calE$ one has $(w,v)\notin\calE$.
\end{itemize}
\end{defin}

From now on, if $(v,w),(w,v)\in\calE$ for two vertices $v,w\in\calV$, we will identify these two directed edges and (with an abuse of notation) we will consider them as a single, undirected, edge; while by a directed edge we will mean only an edge $(v,w)\in\calE$ such that $(w,v)\notin\calE$.
We will represent a digraph as follows: special and ordinary vertices are respectively black and white dots, while if $(v,w),(w,v)\in\calE$ we drow a single unoriented arc joining the vertices $v$ and $w$.
For example, the three diagrams
\begin{equation}\label{eq:examples graphs intro}
 \xymatrix@R=1.5pt{&v_1\\ & \bullet\ar@{-}[ddd]  \\ \\ \\ \circ\ar@{-}[r]\ar[uuur]  & \circ\\v_2&v_3}
 \qquad\qquad\qquad
 \xymatrix@R=1.5pt{v_1&v_2 \\ \bullet\ar[r] & \bullet  \\ \\ \\
 \circ\ar@{-}[r]\ar[ruuu]\ar[uuu] & \circ\ar[uuul]\ar@{-}[uuu] \\ v_3&v_4}
 \qquad\qquad\qquad
 \xymatrix@R=1.5pt{v_1&w&v_5\\ \circ\ar[r] &\bullet&\circ\ar[l] \\  \\  \\
 \circ\ar@{-}[uuu]\ar[uuur]\ar@{-}[r] &\circ\ar@{-}[r]\ar[uuu]& \circ\ar[uuul]\ar@{-}[uuu] \\ v_2&v_3&v_4}
 \end{equation}
represent three digraphs, respectively with 3, 4 and 6 vertices.
Observe that a special vertex may be an end of an undirected edge (as $v_1$ in the first diagram), or even the first coordinate of a directed edge (as $v_1$ in the second diagram, or $v_2$ in \eqref{eq:exagraph intro}); while the vertex $w$ in the third diagram (and also $v_1$ in \eqref{eq:exagraph intro}) is a sinkhole.

\begin{rem}\label{rem:non or graph}\rm
Given a digraph $\Gamma=(\calV,\calE)$, set
\[
 |\calE|=\left\{\:\{v,w\}\in\mathcal{P}_2(\calV)\:\mid\:(v,w)\in\calE\text{or }(w,v)\in\calE\:\right\}.
\]
Then the pair $|\Gamma|=(\calV,|\calE|)$ is a graph (in the sense of \cite[\S~1.1]{graph:book}).
Conversely, from a graph $\mathrm{G}=(\calV,\mathrm{E})$ --- where $\mathrm{E}\subseteq\mathcal{P}_2(\calV)$ --- one may construct the digraph $\mathrm{G}_{\mathrm{or}}=(\calV,\mathrm{E}_{\mathrm{or}})$, with
$$\mathrm{E}_{\mathrm{or}}=\left\{\:(v,w),(w,v)\in\calV\times\calV\:\mid\:\{v,w\}\in\mathrm{E}\:\right\}.$$
Clearly, $|\mathrm{G}_{\mathrm{or}}|=\mathrm{G}$; on the other hand, given a digraph $\Gamma=(\calV,\calE)$, one has $|\Gamma|_{\mathrm{or}}=\Gamma$ if, and only if, $\Gamma$ has only undirected edges.
Henceforth, we will identify graphs and digraphs with only undirected edges via the functors $\textvisiblespace_{\mathrm{or}}$ and $|\textvisiblespace|$, and we will call the latter ``undigraphs''.
\end{rem}

%%%%%%%%%
%%%%%%%%%%%%%%%%%%%%%%%%%%%%%%%%%%%%%%%%%%%%%%%%%%%%%%%%%%%%%%%%%%%%%%%%

\subsection{Special digraphs and special-clique digraphs}\label{ssec:special-clique}

\begin{defin}\label{defin:special graph}\rm

Let $\Gamma=(\calV,\calE)$ be a digraph.
\begin{itemize}
 \item[(a)] $\Gamma$ is said to be {special} if every special vertex is a sinkhole.
 \item[(b)] $\Gamma$ is said to be {special-clique} if it is special and moreover the star of every special vertex is a clique of $\Gamma$.
\end{itemize} 
\end{defin}

The former definition was introduced in \cite[\S~2.3]{BQW}.
It is straightforward to see that both properties are inherited by induced subdigraphs.

\begin{rem}\label{rem:special}\rm
 \begin{itemize}
  \item[(a)] An undigraph is always special and special-clique, as it has no special vertices (and thus the conditions in Definition~\ref{defin:special graph} are trivially satisfied).
  \item[(b)] If $\Gamma=(\calV,\calE)$ is a special digraph and $v_1,v_2\in\calV$ are two distinct special vertices, then $v_1,v_2$ are disjoint --- namely, $(v_1,v_2),(v_2,v_1)\notin \calE$.
  \item[(c)] A complete digraph is special (and special-clique) if, and only if, it has at most a special vertex, and such a vertex is a sinkhole.
 \end{itemize}
\end{rem}

For example, the first two digraphs represented in \eqref{eq:examples graphs intro} are not special --- and thus neither special-clique ---, as their directed edges are not sinkholes. On the other hand, the third diagram represents a
special digraph, as the only special vertex is a sinkhole, but it is not special-clique, as the star of the only special vertex (which is the whole digraph) is not complete.

\begin{exam}\label{exam:special square}\rm
Consider the digraphs with geometric representations
\begin{equation}\label{eq:examples graphs square}
 \xymatrix@R=1.5pt{\circ\ar[r]\ar[ddd] & \bullet  \\ \\ \\ \bullet\ar[uuur]  & \circ\ar[uuu]\ar[l]}
 \qquad\qquad\qquad
 \xymatrix@R=1.5pt{\circ\ar[r] & \bullet  \\ \\ \\
 \circ\ar@{-}[r]\ar[ruuu]\ar@{-}[uuu] & \circ\ar[uuu]\ar@{-}[l]} 
 \qquad\qquad\qquad
  \xymatrix@R=1.5pt{\circ\ar[r]\ar@{-}[dddr]\ar[ddd] & \bullet  \\ \\ \\ \bullet  & \circ\ar[uuu]\ar[l]}
 \end{equation}
 The left-one is not special as the bottom special vertex is not a sinkhole; the center-one is special but not special-clique as the star of the only special vertex is not a clique; the right-one is special-clique.
\end{exam}

It is easy to see that a digraph $\Gamma=(\calV,\calE)$ is special if, and only if, it contains no induced subdigraphs with three vertices whose geometric representation is
\begin{equation}\label{eq:subdigraphs no special}
 \xymatrix@R=1.5pt{& \bullet\ar[ddr] &  \\  \\ \circledast\ar@{--}[rr]\ar[uur]&  & \bullet}
 \qquad\text{or}\qquad
 \xymatrix@R=1.5pt{& \bullet\ar@{-}[ddr] &   \\ \\ \circledast\ar@{--}[rr]\ar[uur]&  &\circledast} 
\end{equation}
--- no matter whether the bottom vertices are joined or not (here we use $\circledast$ to represent vertices which are not necessarily ordinary nor special).
Moreover, a digraph $\Gamma=(\calV,\calE)$ is special-clique if, and only if, it contains no induced subdigraphs with three vertices whose geometric representation is as in \eqref{eq:subdigraphs no special}, nor
\begin{equation}\label{eq:subdigraph no specialclique}
  \xymatrix@R=1.5pt{& \bullet &  \\  \circ\ar[ur]&  & \circ\ar[ul]}
\end{equation}

%%%%%%%%%
%%%%%%%%%%%%%%%%%%%%%%%%%%%%%%%%%%%%%%%%%%%%%%%%%%%%%%%%%%%%%%%%%%%%%%%%

\subsection{Patching of digraphs}\label{ssec:patching}

a digraph $\Gamma=(\calV,\calE)$ is said to be the {patching} of two induced subdigraphs $\Gamma_1,\Gamma_2$ along a common subdigraph $\Gamma'$ if there are three induced subdigraphs $\Gamma_1=(\calV_1,\calE_1)$, $\Gamma_2=(\calV_2,\calE_2)$ and $\Gamma'=(\calV',\calE')$ of $\Gamma$ such that
$$\calV=\calV_1\cup\calV_2,\qquad\calV'=\calV_1\cap\calV_2, \qquad\calE=\calE_1\cup\calE_2$$
(cf. \cite[pp.~4--5]{BQW}).

Not every digraph may be constructed as the patching of two proper induced subdigraphs, as, for example, one has the following fact (whose proof is left to the reader).

\begin{fact}\label{lemma:patching complete}
 Let $\Gamma=(\calV,\calE)$ be a complete digraph. Then $\Gamma$ may not be constructed as the patching of two proper induced subdigraphs.
\end{fact}

\begin{exam}\rm
\begin{itemize}
 \item[(a)] The digraph represented in \eqref{eq:exagraph intro} is the patching of the induced subdigraphs with vertices respectively $\{v_1,v_2,v_3\}$ and $\{v_1,v_3,v_4\}$, along the common subdigraph with vertices $\{v_1,v_3\}$.
 \item[(b)] The third digraph represented in \eqref{eq:examples graphs intro} is the patching of the two ``square'' induced subdigraphs --- with vertices respectively $\{w,v_1,v_2,v_3\}$ and $\{w,v_3,v_4,v_5\}$ ---, which are equal, along the common subdigraph with vertices $\{w,v_3\}$; while the second digraph represented in \eqref{eq:examples graphs intro} is not the patching of any pair of induced subdigraphs, as it is complete (cf. Lemma~\ref{lemma:patching complete}).
 \item[(c)] Each of the three graphs in Example~\ref{exam:special square} is the patching of the two ``triangles'' subdigraphs along the subdigraph which consists of the diagonal of the square.
\end{itemize}
\end{exam}

Let $\Gamma=(\calV,\calE)$ be a special-clique digraph, and let $\Gamma_0=(\calV_0,\calE_0)$ be the induced subdigraph of $\Gamma$ whose vertices $\calV_0$ are all the ordinary vertices of $\Gamma$ --- roughly speaking, $\Gamma_0$ is the maximal unoriented induced subdigraph of $\Gamma$.
It follows from the definition of special-clique digraphs that $\Gamma$ may be constructed by iterating the patching of $\Gamma_0$ with $\mathrm{St}(v)=(\calV_v,\calE_v)$, for every special vertex $v$ of $\Gamma$, along the common clique with vertices $\calV_v\cap\calV_0$.

\begin{exam}\label{exam:special-clique}\rm
The digraph $\Gamma=(\calV,\calE)$ with geometric representation
\[ \xymatrix@R=1.5pt{w_1 &u_1 & u_2 & w_2\\
 \bullet & \circ\ar[l] \ar@{-}[dddl]\ar@{-}[r]\ar@{-}[ddd] & \circ\ar[r] & \bullet \\ \\ \\
 \circ\ar[uuu]\ar@{-}[r] & \circ\ar[uuul]\ar@{-}[r] & \circ\ar@{-}[uuu]\ar[ruuu]\ar[r] & \bullet\\
 u_3&u_4&u_5&w_3 }\]
  is special-clique.
Let $\Gamma_0=(\calV_0,\calE_0)$ be the induced subdigraph with vertices $\calV_0=\{u_1,\ldots,u_5\}$, and let $\Delta_1,\Delta_2,\Delta_3$ be the cliques of $\Gamma$ (and of $\Gamma_0$) with vertices respectively $\{u_1,u_3,u_4\}$, $\{u_2,u_5\}$ and $\{u_5\}$.
Then $\Gamma$ may be constructed by patching $\Gamma_0$ with $\mathrm{St}(w_1)$ along $\Delta_1$; and then patching the resulting digraph with $\mathrm{St}(w_2)$ along $\Delta_2$; and finally patching the resulting digraph with $\mathrm{St}(w_3)$ along $\Delta_3$.
\end{exam}

%%%%%%%%%
%%%%%%%%%%%%%%%%%%%%%%%%%%%%%%%%%%%%%%%%%%%%%%%%%%%%%%%%%%%%%%%%%%%%%%%%

\subsection{The exterior Stanley-Reisner $\F_p$-algebra}\label{ssec:stanreis}

Let $\Gamma=(\calV,\calE)$ be a digraph, and let $V^\ast$ be the $\F_p$-vector space with basis $\calV^\ast=\{v^\ast\colon \calV\to \F_p\:\mid\:v\in\calV\}$, where
\[
 v^\ast(u)=\begin{cases} 1 & \text{if }u=v,\\0&\text{if }u\neq v.\end{cases}
\]
The {exterior Stanley-Reisner $\F_p$-algebra} $\mathbf{\Lambda}_\bullet(\Gamma^\ast)$ associated to $\Gamma$ is the graded $\F_p$-algebra
 \begin{equation}\label{eq:Stanley}
    \mathbf{\Lambda}_\bullet(\Gamma^\ast)=\dfrac{\mathbf{\Lambda}_\bullet(V^\ast)}
  {\left(v^\ast\wedge w^\ast\:\mid\:\{v,w\}\not\in|\calE|\right)},
 \end{equation}
 where $\bfLam_\bullet(V^\ast)$ denotes the exterior $\F_p$-algebra generated by $V^\ast$.
 It is easy to see that for each $n\geq0$, the space $\Lambda_n(\Gamma^\ast)$ has a basis which is in 1-to-1 correspondence with the cliques of $\Gamma$ with $n$ vertices --- e.g., for $n=2$ a basis is given by $\{v^\ast\wedge w^\ast\mid \{v,w\}\in|\calE|\}$.
 In particular, $\Lambda_n(\Gamma^\ast)$ is trivial if $\Gamma$ has less than $n$ vertices.

 \begin{rem}\label{rem:Stanley-Reisner}\rm
  The above definition of the exterior Stanley-Reisner $\F_p$-algebra associated to a digraph is consistent with the definition of the exterior Stanley-Reisner $\F_p$-algebra associated to a graph (see, e.g., \cite[\S~3.2]{papa:raags}).
  In other words, if $\Gamma$ is an undigraph, then its associated exterior Stanley-Reisner $\F_p$-algebra as a digraph, and the exterior Stanley-Reisner $\F_p$-algebra associated to the graph $|\Gamma|$, are equal.
 \end{rem}

 The exterior Stanley-Reisner $\F_p$-algebra is a {quadratic} graded $\F_p$-algebra, which means that the whole structure as a graded $\F_p$-algebra is determined by the spaces of degree 1 and 2 (cf., e.g., \cite[\S~2.1]{QSV}), as the ideal we quotient on in \eqref{eq:Stanley} is generated by elements of degree 2.

%%%%%%%%%%%%%%%%%%%%%%%%%%%%%%%55

%%%%%%%%%%%%%%%%%%%%%%%%%%%%%%%%%%%%%%%%%%%%%%%%%%%%%%%%%%%%%%%%%%%%%%%%5
%%%%%%%%%
%%%%%%%%%%%%%%%%%%%%%%%%%%%%%%%%%%%%%%%%%%%%%%%%%%%%%%%%%%%%%%%%%%%%%%%%

\section{Oriented pro-$p$ RAAGs}
\label{sec:RAAGs}

\subsection{Notation}\label{ssec:not}
Given an arbitrary group $H$, the commutator of two elements $x,y\in H$ is
$$[x,y]=xyx^{-1}y^{-1}.$$
Given three elements $x,y,z\in H$, one has the well-known equalities
\begin{equation}\label{eq:comm times}
 [xy,z]=[x,[y,z]][y,z][x,x]\qquad\text{and}\qquad
 [x,yz]=[x,y][y,[x,z]][x,z].
\end{equation}
from which one deduces
\begin{equation}\label{eq:comm power}
 \left[x^n,y\right]=[x,y]^n\qquad\text{and}\qquad \left[x,y^n\right]=[x,y]^n
\end{equation}
if $x,y\in H$ satisfy, respectively, $[x,[x,y]]=1$ and $[[x,y],y]=1$.

Given a pro-$p$ group $G$, every subgroup will be implicitly assumed to be {closed}, and the generators of $G$, or of a closed subgroup, are to be intended in the topological sense.
In particular,
\begin{itemize}
 \item[(a)] $G'=[G,G]$ will denote the closed derived subgroup of $G$, namely, $G'$ is the closed subgroup of $G$ generated by the commutators $[x,y]$ with $x,y$ running through all elements of $G$;
 \item[(b)] for every $n\geq1$, $G^{p^n}$ will denote the closed subgroup generated by the elements $x^{p^n}$, with $x$ running through all elements of $G$;
 \item[(c)] the {Frattini subgroup} $\Phi(G)$ of $G$ is the closed subgroup of $G$ generated by $G'$ and $G^p$.
\end{itemize}
All the above subgroups are normal subgroups; in particular the quotient $G/\Phi(G)$ is a $\F_p$-vector space.

%%%%%%%%%%%%%%%%%%%%%%%%%%%%%%%%%%%%%%

\subsection{Oriented pro-$p$ RAAGs}\label{ssec:or}

From now on, $q$ will denote a power $p^f$ of the prime number $p$, with $f\geq1$ --- we will tacitly assume that $f\geq2$ in case $p=2$.

Given a digraph $\Gamma=(\calV,\calE)$, let $F_{\calV}$ be the free pro-$p$ group generated by $\calV$.
Given $q$, let $R_{\calE,q}$ be the normal subgroup of $F$ generated as a normal subgroup by the elements $r_{v,w}$ with $\{v,w\}$ running through the elements of $|\calE|$, with
\[
 r_{v,w}=\begin{cases} [v,w] & \text{if }(v,w)\text{ is an undirected edge}, \\
 [w,v]v^{-q} & \text{if }(v,w)\text{ is a directed edge}.
           \end{cases}\]
The oriented pro-$p$ RAAG associated to $\Gamma$ and to $q$ is $G=F_{\calV}/R_{\calE,q}$.
Observe that if $\Gamma$ is an undigraph, then $G$ coincides with the pro-$p$ RAAG associated to the graph $|\Gamma|$, i.e., the pro-$p$ completion of the discrete RAAG associated to $|\Gamma|$ (and obviously it does not depend on the choice of $q$).

\begin{rem}\label{rem:minimal presentation}\rm
 The normal subgroup $R_{\calE,q}$ is contained in $\Phi(F_{\calV})$, whence $G=F_{\calV}/R_{\calE,q}$ is a minimal presentation of $G$ (cf., e.g., \cite[Ch.~III, \S~9, pp.~224--226]{nsw:cohn}).
 In particular, $\calV$ is a minimal generating set of $G$.
\end{rem}

\begin{exam}\label{ex:complete RAAGs}\rm
Let $\Gamma=(\calV,\calE)$ be a complete special digraph, with $\calV=\{w,v_1,\ldots,v_d\}$, $d\geq1$ and $w$ the only special vertex (cf. Remark~\ref{rem:special}--(c)).
Then the oriented pro-$p$ RAAG associated to $\Gamma$ and $q$ is the pro-$p$ group
\[
 G=\left\langle\:w,v_1,\ldots,v_d\:\mid\:[v_i,v_j]=1,\:wv_iw^{-1}=v_i^{1+q}\:\forall\: i=1,\ldots,d\:\right\rangle\simeq \Z_p^d\rtimes \Z_p.
\]
\end{exam}

\begin{rem}\label{rem:orientation RAAG}\rm
 In \cite{BQW}, the oriented pro-$p$ RAAG associated to a digraph $\Gamma$ and to $q$ is defined to be the pair $(G,\theta)$, where $G$ is the pro-$p$ group defined above, and $\theta\colon G\to1+p\Z_p$ is a homomorphism of pro-$p$ groups, the {orientation} of the oriented pro-$p$ RAAG.
 Since we will make no use of such homomorphism, for simplicity we define an oriented pro-$p$ RAAG to be only the pro-$p$ group defined above, without considering the associated homomorphism $\theta$.
\end{rem}

The pro-$p$ RAAGs associated to undigraphs are torsion-free pro-$p$ groups;
in fact, more generally, if $\Gamma$ is an undigraph and $\Gamma'$ is an induced subdigraph of $\Gamma$, the pro-$p$ RAAG associated to $\Gamma'$ is a subgroup of the pro-$p$ RAAG associated to $\Gamma$.
This may not occur for an oriented pro-$p$ RAAGs associated to a digraph, as shown by the following examples.

\begin{exam}\label{ex:mennike}\rm
 Let $\Gamma=(\calV,\calE)$ be the digraph with geometric representation
 \[  \xymatrix@R=1.5pt{& v_2 & \\
& \bullet\ar@/^/[dddr] &  \\ \\ \\  \bullet\ar@/^/[uuur] && \bullet\ar@/^/[ll]  \\  v_1& &v_3 }\]
For $q=p$, the associated oriented pro-$p$ RAAG is
\[
 G=\left\langle\:v_1,v_2,v_3\:\mid\:[v_1,v_2]=v_2^p,[v_2,v_3]=v_3^p,[v_3,v_1]=v_1^p,\:\right\rangle,
\]
which is a finite $p$-group, as shown by J.~Mennike (cf. \cite[ Ch.~I, \S~4.4, Ex.~2(e)]{serre:galc}).
In particular, the oriented pro-$p$ RAAG associated to a single vertex (which is isomorphic to $\Z_p$) and the oriented pro-$p$ RAAG associated to an edge (which is isomorphic to $\Z_p\rtimes \Z_p$, cf. Example~\ref{ex:complete RAAGs}) are not subgroups of $G$ (cf. \cite[Ex.~4.7]{BQW}).
\end{exam}

\begin{exam}\label{ex:torsion}\rm
 Let $\Gamma=(\calV,\calE)$ be the digraph with geometric representation
 \[  \xymatrix@R=1.5pt{& v_1 & \\
& \bullet &  \\ \\ \\  \bullet\ar@/^/[uuur] && \circ\ar@/_/[uuul]\ar@/^/[ll]  \\  v_2& &v_3 }\]
For $q=p$, the associated oriented pro-$p$ RAAG is
\[
 G=\left\langle\:v_1,v_2,v_3\:\mid\:[v_1,v_2]=v_2^p,\:[v_1,v_3]=[v_2,v_3]=v_3^p\:\right\rangle.
\]
Then $[v_2^q,v_3]=v_3^{(1+q)^q-1}$; on the other hand, one computes
\[\begin{split}
\left[[v_2^q,v_3]\right]=\left[v_1v_2v_1^{-1}v_2^{-1},v_3\right] &=
v_1\left(v_2\left(v_1^{-1}\left(v_2^{-1}v_3v_2\right)v_1\right)v_2^{-1}\right)v_1^{-1}\cdot v_3^{-1} \\
&= \left(\left(\left(v_3^{(1+q)^{-1}}\right)^{\epsilon(1+q)^{-1}}\right)^{1+q}\right)^{\epsilon(1+q)}\cdot v_3^{-1}\\
&=v_3^1\cdot v_3^{-1}=1.
\end{split}\]
Since $(1+q)^q-1=(1+q^2+q^3(q-1)/2+\ldots)-1\neq0$, the vertex $v_3$ yields non-trivial torsion.
In particular, the oriented pro-$p$ RAAG associated to the digraph consisting of the single vertex $v_3$ is not a subgroup of $G$.
\end{exam}

Observe that the digraphs in Examples~\ref{ex:mennike}--\ref{ex:torsion} are not special.
Indeed, for special digraphs one has the following result (cf. \cite[Prop.~4.11]{BQW}).

\begin{prop}\label{prop:special clique inclusion}
 Let $\Gamma=(\calV,\calE)$ be a special digraph, and let $G$ be the oriented pro-$p$ RAAG associated to $\Gamma$ and to $q$.
 If $\Delta=(\calV_\Delta,\calE_\Delta)$ is a clique of $\Gamma$ and $G_\Delta$ the oriented pro-$p$ RAAG associated to $\Delta$, then the inclusion $\calV_\Delta\hookrightarrow\calV$
 induces a monomorphism of pro-$p$ groups $G_\Delta\hookrightarrow G$.
\end{prop}

Consequently, if $\Gamma=(\calV,\calE)$ is a special-clique graph, and $G$ is an associated oriented pro-$p$ RAAG,
then for every special vertex $v\in\calV$ both the oriented pro-$p$ RAAG $G_{\mathrm{St}(v)}$ associated to the star $\mathrm{St}(v)=(\calV_v,\calE_v)$, and the oriented pro-$p$ RAAG $A$ associated to the clique with vertices $\calV_v\smallsetminus\{v\}$ (which is a free abelian pro-$p$ group), are subgroups of $G$.

\begin{rem}\label{rem:generalized}\rm
In \cite[\S~5]{QSV}, the authors introduced the notion of {generalized $p$-RAAGs} associated to {$p$-labeled} oriented graphs.
The family of oriented pro-$p$ RAAGs associated to digraphs is a subfamily of the family of generalized $p$-RAAGs (cf. \cite[\S~4.2]{BQW}), thus all results on generalized $p$-RAAGs obtained in \cite{QSV} apply also to oriented pro-$p$ RAAGs associated to digraphs.
\end{rem}

%%%%%%%%%%%%%%%%%%%%%%%%%%%%%%%%%%%%%%%%%%%%%%%%%%%%%%%%%%%%%%%%%%%%%%%%5
%%%%%%%%%%%%%%%%%%%%%%%%%%%%%%%%%%%%%%%%%%%%%%%%%%%%%%%%%%%%%%%%%%%%%%%%5
\subsection{Oriented pro-$p$ RAAGs, free products and amalgams}\label{ssec:or prop}

Let $\Gamma=(\calV,\calE)$ be a digraph, and for $q$ a $p$-power let $G$ be the associated pro-$p$ RAAG.
It follows from the definition of $G$ that if $\Gamma$ has connected components $\Gamma_1,\ldots,\Gamma_r$, then
$G$ is the free pro-$p$ product of the oriented pro-$p$ RAAGs $G_1,\ldots,G_r$ associated respectively to the digraphs $\Gamma_1,\ldots,\Gamma_r$, i.e., $G=G_1\amalg\ldots \amalg G_r$ (for the definition of free pro-$p$ product of pro-$p$ groups see, e.g., \cite[\S~9.1]{ribzal}).

For patching of digraphs we may have a similar phenomenon.
Recall that the the {amalgamated free pro-$p$ product} $G$ of two pro-$p$ groups $G_1,G_2$ with {amalgam} a common subgroup $H\subseteq G_1,G_2$, is the push-out
\[\xymatrix@R=1.5pt{& G_1\ar@{-->}[ddr]^{\psi_1} & \\ \\ H\ar@{^{(}->}[ruu]\ar@{^{(}->}[rdd] && G \\ \\ &G_2\ar@{-->}[uur]_{\psi_2}& }
\]
in the category of pro-$p$ groups, which is unique (cf. \cite[\S~9.2]{ribzal}).
We write $G= G_1\amalg_H G2$.
An amalgamated free pro-$p$ product $G= G_1\amalg_H G_2$ is said to be proper if the homomorphisms $\psi_1,\psi_2$ are monomorphisms. In that case one identifies $G_1,G_2,H$ with their images in $G$.

In case of special digraphs one has the following result, which is the ``oriented version'' of \cite[Prop.~5.22]{QSV}, which deals with a more general case (cf. Remark~\ref{rem:generalized}).

\begin{lem}\label{prop:amalg}
 Let $\Gamma_1=(\calV_1,\calE_1)$ and $\Gamma_2=(\calV_2,\calE_2)$ be two special digraphs, and let $\Delta=(\calV_\Delta,\calE_\Delta)$ be a clique of both $\Gamma_1$ and $\Gamma_2$ whose vertices are ordinary.
 Moreover, for $q=p^f$ let $G_1,G_2,G_\Delta$ be the oriented pro-$p$ RAAGs associated respectively to $\Gamma_1,\Gamma_2,\Delta$.
 Then $G_\Delta$ is a subgroup of both $G_1$ and $G_2$, and the amalgamated free pro-$p$ product
 \[
G=  G_1\amalg_{G_\Delta}G_2
 \]
is proper.
\end{lem}

\begin{proof}
 Since $\Delta$ is a complete undigraph, $G_\Delta$ is the free abelian pro-$p$ group generated by the set of vertices $\calV_\Delta$.
 Moreover, $G_\Delta$ is a subgroup of both $G_1,G_2$ by Proposition~\ref{prop:special clique inclusion}.

Now fix $i=1$ or $i=2$, and let $\calV_0$ be the set of ordinary vertices of $\Gamma_i$ --- thus $\calV_0\supseteq\calV_\Delta$ ---, and let $\mathcal{Y}$ be the set $\calY=\{y_u\mid u\in\calV_0\}$.
Moreover, let $A$ be the free abelian pro-$p$ group generated by $\calY$, and let $\bar G$ be the pro-$p$ group
\[
 \bar G=A\rtimes X=
 \left\langle\:\calY,x^\epsilon\:\mid
 \:[y_{u},y_{u'}]=1,[x,y_u]=y_u^{\epsilon q}\:\forall\:u,u'\in\calV_0\:\right\rangle,
\]
where $X$ is either trivial --- and hence $\epsilon=0$ --- if $\Gamma_i$ has no special vertices; or $X$ is the pro-$p$-cyclic group generated by an element $x$ such that $xyx^{-1}=y^{1+q}$ for every $y\in A$ --- and hence $\epsilon=1$ --- if $\Gamma_i$ has at least a special vertex.
By construction, one has
\begin{equation}\label{eq:barG pn}
\bar G^{p^n}=A^{p^n}\rtimes X^{p^n}.
\end{equation}
(Observe that $\bar G$ is isomorphic to the oriented pro-$p$ RAAG associated to a complete special digraph whose ordinary vertices are $\calY$, and whose only special vertex --- if $X$ is not trivial --- is $x$, cf. Example~\ref{ex:complete RAAGs}.)
Now, the assignment
\[
 u\longmapsto y_u\qquad\text{and}\qquad w\longmapsto x
\]
for every $u\in\calV_0$ and for every special vertex $w\in\calV$ --- if there is any, and thus if $X\neq\{1\}$ --- induces a homomorphism of pro-$p$ groups $\pi\colon G_i\to \bar G$, as
\[
 \pi([u,u'])=[y_u,y_{u'}]=1\qquad \text{and}\qquad \pi([w,u])=[x,y_u]=y_u^q=\pi(u^q)
\]
for every $u,u'\in\calV_0$, and for every special vertex $w\in\calV$ (recall that there are no relations between the special vertices of $\Gamma_i$, as they are disjoint, for $\Gamma$ is special).
Moreover, $\pi$ is clearly surjective.
On the other hand, $\pi(G_\Delta)=A_\Delta$, where $A_\Delta\subseteq A$ is the free abelian pro-$p$ group generated by $\{y_u\mid u\in\calV_\Delta\}$.
Thus, the restriction $\pi\vert_{G_\Delta}\colon G_\Delta\to A_\Delta$ is an isomorphism of (abelian) pro-$p$ groups.

We claim that
\begin{equation}\label{eq:inclusion pn}
  G_\Delta^{p^n}=G_\Delta\cap G_i^{p^n}\qquad\text{for every }n\geq1.
\end{equation}
The inclusion $G_\Delta^{p^n}\subseteq G_\Delta\cap G_i^{p^n}$ is obvious.
Conversely, let $z$ be an element of $G_{\Delta}$ such that $z\in G_i^{p^n}$.
Then $\pi(z)\in\bar G^{p^n}$, and since $\pi(z)\in A_\Delta\subseteq A$, by \eqref{eq:barG pn} one has $\pi(z)\in A_\Delta^{p^n}$.
Finally, $$z=\pi^{-1}(\pi(z))\in G_{\Delta}^{p^n},$$ as $\pi\vert_{G_\Delta}\colon G_\Delta\to A_\Delta$ is an isomorphism of abelian pro-$p$ groups.

Since $i$ was arbitrarily chosen, \eqref{eq:inclusion pn} holds for both $i=1,2$.
Then \cite[Thm.~9.2.4]{ribzal} implies that the amalgamated free pro-$p$ product is proper.
\end{proof}

It may happen that an oriented, but not special, graph decomposes as patching of two induced subdigraphs, but the associated oriented pro-$p$ RAAG is not the amalgamated free pro-$p$ product of the two oriented pro-$p$ RAAGs associated to the two induced subdigraphs, as shown by the following.

\begin{exam}\rm
 Let $\tilde\Gamma=(\tilde\calV,\tilde\calE)$ be the digraph with geometric representation \eqref{eq:exagraph intro} --- namely,
\[
  \xymatrix@R=1.5pt{&v_1& \\&\bullet& \\  \\  \\
 \bullet\ar[uuur] &\circ\ar[l]\ar@{-}[r]\ar[uuu]& \circ\ar[uuul] \\ v_2&v_3&v_4}
\]
---, and let $\Gamma=(\calV,\calE)$, $\Gamma'=(\calV',\calE')$ and $\Delta=(\calV_\Delta,\calE_\Delta)$ be the induced subdigraphs of $\tilde\Gamma$ with vertices respectively
\[
 \calV=\{v_1,v_2,v_3\},\qquad \calV'=\{v_1,v_3,v_4\},\qquad\calV_\Delta=\{v_2,v_3\}.
\]
Hence $\Gamma$ is as in Example~\ref{ex:torsion}, and $\tilde \Gamma$ is the patching of $\Gamma$ and $\Gamma'$ along $\Delta$.

\noindent
For $q=p$ let $\tilde G$, $G$, $H$ and $G_\Delta$ be the oriented pro-$p$ RAAGs associated respectively to $\tilde\Gamma$, $\Gamma$, $\Gamma'$ and $\Delta$.
Hence $G$ is as in Example~\ref{ex:torsion}, while $H\simeq\Z_p^2\rtimes \Z_p$ and $G_\Delta\simeq\Z_p\rtimes\Z_p$ (cf. Example~\ref{ex:complete RAAGs}).
Then $\tilde G$ is not an amalgamated free pro-$p$ product of $G$ and $H$: the subgroup of $H$ generated by $v_1,v_3$ is $G_\Delta$ by Proposition~\ref{prop:special clique inclusion}, and it is torsion-free, while the subgroup of $G$ generated by $v_1,v_3$ has non-trivial torsion by Example~\ref{ex:torsion}.
\end{exam}

%%%%%%%%%%%%%%%%%%%%%%%%%%%%%%%%%%%%%%%%%%%%%%%%%%%%%%%%%%%%%%%%%%%%%%%%5
%%%%%%%%%
%%%%%%%%%%%%%%%%%%%%%%%%%%%%%%%%%%%%%%%%%%%%%%%%%%%%%%%%%%%%%%%%%%%%%%%%

%%%%%%%%%%%%%%%%%%%%%%%%%%%%%%%%%%%%%%%%%%%%%%%%%%%%%%%%%%%%%%%%%%%%%%%%

\section{The $\F_p$-cohomology of oriented pro-$p$ RAAGs}
\label{sec:H}

%%%%%%%%%%%%%%%%%%%%%%%%%%%%%%%%%%%%%%%%%%%%%%%%%%%%%%%%%%%%%%%%%%%%%%%%5
%%%%%%%%%

\subsection{$\F_p$-cohomology of pro-$p$ groups in a nutshell}\label{ssec:cohom}

Throughout this subsection, $G$ will denote a pro-$p$ group, and the finite field $\F_p$ will be considered as a trivial $G$-module.
For $n\geq0$, we will denote the $n$th $\F_p$-cohomology group $\rmH^n(G,\F_p)$ simply by $\rmH^n(G)$.
For the reader's convenience, in this subsection (which may be safely skipped by who is accustomed with group cohomology) we recall briefly some basic facts on the cohomology of $G$ with coefficients in $\F_p$ --- for a more detailed account see \cite[Ch.~I, \S\S~2--4]{serre:galc} or \cite[Ch.~I]{nsw:cohn}, to which we make reference.

\subsubsection{Degree 0 and 1}
The 0th $\F_p$-cohomology group $\rmH^0(G)$ is just $\F_p$, while one has an equality of discrete $\F_p$-vector spaces
\begin{equation}\label{eq:H1}
 \rmH^1(G)=\mathrm{Hom}(G,\F_p),
\end{equation}
where the former is the group of homomorphisms of pro-$p$ groups $G\to\F_p$, where the latter is considered as a cyclic group of order $p$, cf. \cite[Ch.~I, \S~4.2]{serre:galc}.

If $G$ is finitely generated, then \eqref{eq:H1} induces an isomorphism of finite $\F_p$-vector spaces $$\rmH^1(G)\simeq(G/\Phi(G))^\ast,$$ and $\textvisiblespace^\ast$ denotes the $\F_p$-dual.
In particular, the dimension of $\rmH^1(G)$ as a $\F_p$-vector space is equal to the minimum number of generators of $G$ (cf. \cite[p.~31]{serre:galc}).

\subsubsection{Higher degrees}
For $n\geq1$, the $n$th $\F_p$-cohomology group $\rmH^n(G)$ is a subquotient of the discrete $\F_p$-vector space
\[
 C^n(G)=\left\{\:f\colon \underbrace{G\times \ldots\times G}_{n\text{times}}\to\F_p\:\mid\: f\text{is a continuous map}\:\right\},
\]
where the Cartesian product $G\times\ldots\times G$ has the product topology, and $\F_p$ is a discrete space (cf. \cite[Ch.~I, \S~2.2]{serre:galc}).
(If $n=0$, $\rmH^0(G)=\F_p$ may be seen as the space of constant functions.)

\subsubsection{Cup-product}
Pick two elements $\alpha_1\in\rmH^{n_1}(G)$ and $\alpha_2\in\rmH^{n_2}(G)$, for some $n_1,n_2\geq0$, and let $f_1\in C^{n_1}(G)$ and $f_2\in C^{n_2}(G)$ be representatives respectively of $\alpha_1$ and $\alpha_2$.
Then the continuous map $f_1\cdot f_2\in C^{n_1+n_2}(G)$, defined pointwise via the product of $\F_p$ as a field, belongs to some element of $\rmH^{n_1+n_2}(G)$, which is called the cup-product of $\alpha_1$ and $\alpha_2$, and it is denoted by $\alpha_1\smallsmile\alpha_2$ (cf. \cite[Ch.~I, \S~4]{nsw:cohn}).
Moreover, one has
\[
 \alpha_2\smallsmile\alpha_1=(-1)^{n_1n_2}\alpha_1\smallsmile\alpha_2
\]
(cf. \cite[Prop.~1.4.4]{nsw:cohn}).
Altogether, the space $\bfH^\bullet(G)=\coprod_{n\geq0}\rmH^n(G)$ endowed with the cup-product, is a graded, and graded-commutative, $\F_p$-algebra.

\subsubsection{Functoriality}\label{sssec:funct}
Let $\psi\colon G_1\to G_2$ be a homomorphism of pro-$p$ groups.
In cohomology, $\psi$ induces a homomorphism $\psi^n\colon\rmH^2(G_2)\to\rmH^n(G_1)$ for every $n\geq1$, where $$\psi^1(\alpha)=\alpha\circ\psi\colon G_1\longrightarrow\F_p$$ for every $\alpha\in\rmH^1(G_2)$ (cf. \cite[Ch.~I, \S~2.4]{serre:galc}).
If $\psi$ is surjective, then one has an exact sequence
\begin{equation}\label{eq:5tes}
   \begin{tikzpicture}[descr/.style={fill=white,inner sep=2pt}]
        \matrix (m) [
            matrix of math nodes,
            row sep=3em,
            column sep=3em,
            text height=1.5ex, text depth=0.25ex
        ]
        { 0 &  \rmH^1(G_2)& \rmH^1(G_1) & \rmH^1(\Ker(\psi))^{G_1}\\
            & \rmH^2(G_2) &  \rmH^2(G_1) &   \\
           };

        \path[overlay,->, font=\scriptsize,>=latex]
        (m-1-1) edge node[auto] {} (m-1-2)
        (m-1-2) edge node[auto] {$\psi^1$} (m-1-3)
        (m-1-3) edge node[auto] {} (m-1-4)
        (m-1-4) edge [out=355,in=175] node[auto] {} (m-2-2)
        (m-2-2) edge node[auto] {$\psi^2$} (m-2-3);
\end{tikzpicture}
\end{equation}
where $\rmH^1(\Ker(\psi))^{G_1}$ denotes the subspace of $\rmH^1(\Ker(\psi))$ of the elements fixed by the action given by
$g.\alpha(h)=\alpha(g^{-1}hg)$ for every $g\in G_1$, $h\in\Ker(\psi)$ and $\alpha\in\rmH^1(\Ker(\psi))$ (cf. \cite[Ch.~I, \S~2.6--(b)]{serre:galc}).

On the other hand, if $H=G_1$ is a subgroup of $G=G_2$, and $\psi\colon H\to G$ is the inclusion, for every $n\geq1$ the map $\psi^n$ is called the restriction (of degree $n$), and denoted by $\res_{G,H}^n$ (cf. \cite[p.~12]{serre:galc}).
If $\alpha\in\rmH^1(G)=\mathrm{Hom}(G,\F_p)$, we write simply $\res_{G,H}^1(\alpha)=\alpha\vert_H$.
For every $\alpha_1\in\rmH^{n_1}(G)$ and $\alpha_2\in\rmH^{n_2}(G)$, with $n_1,n_2\geq0$ and $n=n_1+n_2$, one has
\begin{equation}\label{eq:res cup}
 \res_{G,H}^{n}(\alpha_1\smallsmile\alpha_2)=\res_{G,H}^{n_2}(\alpha_1)\smallsmile\res_{G,H}^{n_2}(\alpha_2)
\end{equation}
(cf. \cite[p.~15]{serre:galc}).

%%%%%%%%%%%%%%%%%%%%%%%%%%%%%%%%%%%
\subsubsection{Degree 2}\label{sssec:H2}
Let $G=F/R$ be a minimal presentation of $G$, and suppose that $G$ (and thus also $F$) is finitely generated, and $R$ is a finitely generated as a normal subgroup of $F$.
By \eqref{eq:H1} the canonical projection $\pi\colon F\to G$ yields an isomorphism in cohomology $\pi^1\colon\rmH^1(G)\simeq \rmH^1(F)$; on the other hand $\rmH^2(G)$ is trivial as $F$ is free (cf. \cite[Ch.~I, \S~3.4, Cor.~p.~23]{serre:galc}).
Finally, the vector space $\rmH^1(R)^F$ is isomorphic to $(R/R^p[R,F])^\ast$.
Altogether, \eqref{eq:5tes} applied to $\psi=\pi$ implies that $\pi^2$ is an isomorphism, too, and thus one has an isomorphism
\begin{equation}\label{eq:H2}
{\mathrm{trg}\colon \left(\dfrac{R}{R^p[R,F]}\right)^\ast\overset{\sim}{\longrightarrow} \rmH^2(G) },
\end{equation}
called transgression (cf. \cite[Ch.~I, \S~4.3]{serre:galc}).
In particular, the dimension of $\rmH^2(G)$ as a $\F_p$-vector space is equal to the minimal number of generators of $R$ as a normal subgroup of $F$.

%%%%%%%%%%%%%%%%%%%%%%%%%%%%%%%%%%%%%%%%%%%%55
%%%%%%%%%%5
%%%%%%%%%%%%%%%%%%%%%%%%%%%%%%%%%%%%%%%%%%%%%%%%%%%%%%5

%%%%%%%%%%%%%%%%%%%%%%%%%%%%%%%%%%%%%%%%%%%%%%%%%%%%%%%%%%%%%%%%%%%%%%%%5
%%%%%%%%%

\subsection{$\F_p$-cohomology of oriented pro-$p$ RAAGs}\label{ssec:cohom RAAGS}

Let $G$ be the oriented pro-$p$ RAAG associated to a digraph $\Gamma=(\calV,\calE)$ and to $q$.
 Since $G$ is minimally generated by $\calV$ (cf. Remark~\ref{rem:minimal presentation}), for every $v\in \calV$ the map $v^\ast$ induces a homomorphism of pro-$p$ groups $G\to\F_p$, which we will denote (with a slight abuse of notation) by $v^\ast$ as well.
 Then by \eqref{eq:H1} one has
 \begin{equation}\label{eq:H1Lam1}
\rmH^1(G)=\Lambda_1(\Gamma^\ast)=V^\ast.
 \end{equation}
This equality extends to an isomorphism
\begin{equation}\label{eq:H2Lam2}
 \lambda_2 \colon \Lambda_2(\Gamma^\ast)\overset{\sim}\longrightarrow \rmH^2(G),
\end{equation}
where $\lambda_2(v^\ast\wedge w^\ast)=v^\ast\smallsmile w^{\ast}$ (cf. \cite[Lemma~5.8]{QSV}) --- observe that
$$\dim\left(\rmH^2(G)\right)=\mathrm{card}(|\calE|)=\dim\left(\Lambda_2(\Gamma^\ast)\right).$$
In particular, if one considers the minimal presentation $G=F_{\calV}/R_{\calE,q}$, then the $\F_p$-vector space ${R_{\calE,q}}/N$ --- where we set $N=(R_{\calE,q})^p[R_{\calE_q},F_\calV]$ --- has a basis
$\{r_{v,w}N\:\mid\:\{v,w\}\in|\calE|\}$, and the isomorphism $\lambda_2$ and the transgression \eqref{eq:H2} are described by
\[
\xymatrix@C=60pt{v^\ast\wedge w^\ast \ar@{|->}@/^1.5pc/[r]^-{\lambda_2} & v^\ast\smallsmile w^\ast  \ar@{|->}@/^1.5pc/[r]^-{\mathrm{trg}^{-1}}  \ar@{|->}@/^1.5pc/[l]^-{\lambda_2^{-1}}  &
\pm\left(r_{v,w} N\right)^\ast \ar@{|->}@/^1.5pc/[l]^-{\mathrm{trg}} } \]
for every $\{v,w\}\in|\calE|$ (the sign of the right-side term is minus if $(v,w)$ is special, as in this case the commutator showing up in $r_{v,w}$ is $[w,v]=[v,w]^{-1}$).

\begin{rem}\label{rem:res lambda}\rm
Let $\Gamma=(\calV,\calE)$ be a digraph, and let $G$ be the oriented pro-$p$ RAAG associated to $\Gamma$ and to $q$.
Suppose that $\Gamma$ has an induced subdigraph $\Gamma'=(\calV',\calE')$ such that the oriented pro-$p$ RAAG $H$ associated to $\Gamma'$ and to $q$ is a subgroup of $G$ via the inclusion $\calV'\hookrightarrow\calV$ (e.g., $\Gamma$ is special and $\Gamma'$ is a clique, cf. Proposition~\ref{prop:special clique inclusion}).
Then the map $\res_{G,H}^1\colon \rmH^1(G)\to\rmH^1(H)$ is given by
\[\res_{G,H}^1(v^\ast)=\begin{cases} v^\ast\vert_H & \text{if }v\in\calV',\\0 & \text{if }v\notin\calV';
\end{cases} \]
while the map $\res_{G,H}^2\colon \rmH^2(G)\to\rmH^2(H)$ is given by
\[ \res_{G,H}^2(v^\ast\smallsmile w^\ast)=
 \begin{cases} (v^\ast\vert_H)\smallsmile(w^\ast\vert_H) & \text{if }v,w\in\calV',\\
0 & \text{if }\{v,w\}\notin \calE',
 \end{cases}\]
and its kernel is
\[\begin{split}
   \Ker(\res_{G,H}^2) &=\mathrm{Span}_{\F_p}\left(\:v^\ast\smallsmile w^\ast\:\mid\: \{v,w\}\notin\calE'\:\right)\\
   &=  \rmH^1(G)\smallsmile\Ker(\res_{G,H}^1),
  \end{split}\]
  where the latter is the subspace of $\rmH^2(G)$ generated by the cup-products of elements of the first factor with elements of the second factor.
\end{rem}

By the result of K.~Lorensen \cite[Thm.~2.6]{lorensen}, if $\Gamma$ is an undigraph and $G$ is the associated pro-$p$ RAAG, then \eqref{eq:H1Lam1}--\eqref{eq:H2Lam2} extend to an isomorphism of graded $\F_p$-algebras $\bfLam_\bullet(\Gamma^\ast)\simeq\bfH^\bullet(G)$.
In general, this is not the case for oriented pro-$p$ RAAGs associated to digraphs.

\begin{exam}\rm
It is well-known that if a pro-$p$ group $G$ has non-trivial torsion, then $\rmH^n(G)\neq0$ for every $n\geq0$.
Therefore, if $\Gamma$ is a digraph as in Example~\ref{ex:mennike}, or as in Example~\ref{ex:torsion}, and $G$ is the associated oriented pro-$p$ RAAG, for $q=p^f$, then $\bfLam_\bullet(\Gamma^\ast)\not\simeq\bfH^\bullet(G)$, as $\Lambda_n(\Gamma^\ast)=0$ for $n\geq4$.
\end{exam}

One knows that a digraph $\Gamma=(\calV,\calE)$ yields oriented pro-$p$ RAAGs whose $\F_p$-cohomology algebra is isomorphic to the associated exterior Stanley-Reisner $\F_p$-algebra in the following cases:
\begin{itemize}
 \item[(a)] if $\Gamma$ is an undigraph, as proved by K.~Lorensen;
 \item[(b)] if $\Gamma$ is triangle-free (cf. \cite[Thm.~F]{QSV});
 \item[(c)] if $\Gamma$ is special and chordal (cf. \cite[Thm.~H]{QSV} and \cite[Thm.~1.3--(iii)]{BQW});
 \item[(d)] if $\Gamma$ is obtained by mirroring a digraph yielding an oriented pro-$p$ RAAG (cf. \cite[Rem.~5.25--(b)]{QSV}).
\end{itemize}

\begin{rem}\label{rem:Lam quadratic}\rm
Let $G$ be the oriented pro-$p$ RAAG associated to a digraph $\Gamma=(\calV,\calE)$ and to $q$.
If the $\F_p$-cohomology algebra $\bfH^\bullet(G)$ is a quadratic $\F_p$-algebra (cf. \S~\ref{ssec:stanreis}), then it is necessarily isomorphic to $\bfLam_\bullet(\Gamma^\ast)$, as the latter is {the} quadratic $\F_p$-algebra generated by $\Lambda_1(\Gamma^\ast)=\rmH^1(G)$, and with space of degree 2 $\Lambda_2(\Gamma^\ast)\simeq\rmH^2(G)$ (cf. \cite[Thm.~E]{QSV}).
\end{rem}

%%%%%%%%%%%%%%%%%%%%%%%%%%%%%%%%%%%%%%%%%%%%%%%%%%%%%%%%%%%%%%%%%%%%%%%%5
%%%%%%%%%

\subsection{Cohomology and special-clique digraphs}\label{ssec:H specialclique}

We are ready to prove Theorem~\ref{thm:quad intro}.
We follow the strategy used to prove \cite[Thm.~H]{QSV} (see also \cite[Rem.~2.25--(c)]{QSV}).

\begin{thm}\label{thm:quadratic in text}
 Let $\Gamma=(\calV,\calE)$ be a special-clique special digraph, and for $q=p^f$ let $G$ be the associated oriented pro-$p$ RAAG.
 Then $\bfH^\bullet(G)\simeq\bfLam_\bullet(\Gamma^\ast)$.
\end{thm}

\begin{proof}
We proceed by induction on the number of special vertices of $\Gamma$.
If $\Gamma$ has no special vertices, then $\Gamma$ is an undigraph, and thus $\bfH^\bullet(G)\simeq\bfLam_\bullet(\Gamma^\ast)$ (cf. \S~\ref{ssec:cohom RAAGS}).

Now suppose that $\Gamma$ has $n$ special vertices, $n\geq1$, and that the statement holds for every special-clique special digraph with at most $n-1$ special vertices.

Let $w\in\calV$ be a special vertex, and put
$$\mathrm{St}(w)=(\calV_w,\calE_w),\qquad \Delta=(\calV_\Delta,\calE_\Delta),\qquad \Gamma'=(\calV',\calE'),$$
where $\Delta$ and $\Gamma'$ are the induced subdigraphs of $\Gamma$ whose vertices are respectively $\calV_\Delta=\calV_w\smallsetminus\{w\}$ and $\calV'=\calV\smallsetminus\{w\}$.
Then $\mathrm{St}(s)$ and $\Delta$ are cliques of $\Gamma$, while $\Gamma'$ is a special-clique special digraph with $n-1$ special vertices. Moreover, $\Gamma$ is the patching of $\mathrm{St}(w)$ and $\Gamma'$ along $\Delta$.

By Proposition~\ref{prop:special clique inclusion}, the subgroups $G_w$ and $G_\Delta$ of $G$ generated respectively by $\calV_w$ and $\calV_\Delta$ are isomorphic to the oriented pro-$p$ RAAGs associated respectively to $\mathrm{St}(w)$ and $\Delta$.
Analogously, $G_\Delta$ --- which is the free abelian pro-$p$ group generated by $\calV_{\Delta}$ --- is a subgroup of both $G_w$ and $G_{\Gamma'}$, where $G_{\Gamma'}$ is the oriented pro-$p$ RAAG associated to $\Gamma'$.
By Lemma~\ref{prop:amalg}, the amalgamated free pro-$p$ product
\begin{equation}\label{eq:amalgam proof}
  \tilde G=G_w\amalg_{G_\Delta}G_{\Gamma'}
\end{equation}
is proper.

By induction, the three subgroups $G_w,G_{\Gamma'},G_\Delta$ have quadratic $\F_p$-cohomology, which we identify respectively with $\bfLam_\bullet(\mathrm{St}(w)^\ast)$, $\bfLam_\bullet((\Gamma')^\ast)$ and $\bfLam_\bullet(\Delta^\ast)$.
In order to apply \cite[Thm.~B]{QSV} to the proper amalgamated free pro-$p$ product \eqref{eq:amalgam proof}, we need to check that the restriction maps
\[
 \res_{G_w,G_\Delta}^1\colon \rmH^1(G_w)\to\rmH^1(G_\Delta)\qquad\text{and}\qquad
 \res_{G_{\Gamma'},G_\Delta}^1\colon \rmH^1(G_{\Gamma'})\to\rmH^1(G_\Delta)\]
are surjective; and moreover that
\[\begin{split}
 \Ker(\res_{G_w,G_\Delta}^2)&=\Ker(\res_{G_w,G_\Delta}^1)\wedge \rmH^1(G_w),\\
 \Ker(\res_{G_{\Gamma'},G_\Delta}^2)&=\Ker(\res_{G_{\Gamma'},G_\Delta}^1)\wedge \rmH^1(G_{\Gamma'})
  \end{split}\]
But these condition are satisfied by Remark~\ref{rem:res lambda}, and thus we may apply \cite[Thm.~B]{QSV}, which implies that the $\F_p$-algebra $\bfH^\bullet(\tilde G)$ is quadratic.
In particular, for every $n\geq1$ one has a short exact sequence
\[
 \xymatrix{0\ar[r]&\rmH^n(\tilde G) \ar[r]^-{f_{\tilde G}^n} &
 \Lambda_n(\mathrm{St}(w)^\ast)\oplus \Lambda_n((\Gamma')^\ast)
 \ar[r]^-{f^n_{G_{\Delta}}}  & \Lambda_n(\Delta^\ast)\ar[r]&0},
\]
where $f_{\tilde G}^n=(\res_{\tilde G,G_{w}}^n,\res_{\tilde G,G_{\Gamma'}}^n)$, and
$f^n_{G_{\Delta}}$ extends (via multiplication) to degree $n$ the map $f^1_{G_{\Delta}}(v_1^\ast,v_2^\ast)=v_1^\ast\vert_{G_\Delta}-v_2^\ast\vert_{G_\Delta}$ for $v_1\in\calV_w$ and $v_2\in\calV'$
(cf. \cite[p.~653]{QSV}, see also \cite[Prop.~9.2.13]{ribzal}).
Again by Remark~\ref{rem:res lambda}, $\Ker(f_{\tilde G}^1)$ is easily seen to be isomorphic to $V^\ast=\Lambda_1(\Gamma^\ast)$, as $\calV=\{w\}\cup\calV_\Delta\cup(\calV'\smallsetminus\calV_\Delta)$;
while kernel of $f_{\tilde G}^2$ is
\[
\left(w^\ast\wedge\Lambda_1(\mathrm{St}(w)^\ast)\right)\oplus
\left(\Lambda_2(\Delta^\ast),-\Lambda_2(\Delta^\ast)\right)\oplus
\left(\Ker(\res_{G_{\Gamma'},G_\Delta}^1)\wedge\Lambda_1((\Gamma')^\ast)\right),
\]
which is isomorphic to $\Lambda_2(\Gamma^\ast)$, as
$$|\calE|=\left\{\:\{w,u\}\:\mid\:u\in\calV_w\:\right\}\cup|\calE_\Delta|\cup
\left\{\:\{v,u\}\:\mid\:v\in\calV'\smallsetminus\calV_\Delta,u\in\calV'\:\right\}.$$
Altogether, $\bfH^\bullet(\tilde G)\simeq \bfLam_\bullet(\Gamma^\ast)$ (cf. Remark~\ref{rem:Lam quadratic}).

Finally, we claim that $\tilde G\simeq G$.
Indeed, by the universal property of push-out's, the monomorphism $G_w\to G$ and the homomorphism $G_{\Gamma'}\to G$ induced by $\calV'\hookrightarrow\calV$ yield the commutative diagram
\[\xymatrix@R=1.5pt{& G_w\ar@{^{(}->}[ddr]\ar@/^1pc/[ddrrr] & &&\\ \\
 G_\Delta\ar@{^{(}->}[ruu]\ar@{^{(}->}[rdd] && \tilde G\ar@{-->}[rr]^-{\psi} && G
 \\ \\ &G_{\Gamma'}\ar@{^{(}->}[ruu]\ar@/_1pc/[uurrr] &&& }\]
Consider the homomorphisms $\psi^n\colon \rmH^n(G)\to\rmH^n(\tilde G)$, $n\geq1$ (cf. \S~\ref{sssec:funct}).
Since $\psi^1$ and $\psi^2$ are isomorphisms, $\rmH^1(\Ker(\psi))^{\tilde G}=0$ by \eqref{eq:5tes}.
Hence
\[
0=\left(\rmH^1(\Ker(\psi))^{\tilde G}\right)^\ast\simeq \frac{\Ker(\psi)}{\Ker(\psi)^p\left[\Ker(\psi),\tilde G\right]}
\]
(here the isomorphism is an isomorphism of $\F_p$-vector spaces), and also the latter is trivial.
Therefore, also $\Ker(\psi)$ is trivial, as the above quotient gives a set of generators of $\Ker(\psi)$ as a normal subgroup of $\tilde G$.
\end{proof}

It remains an open problem to determine which digraphs yield oriented pro-$p$ RAAGs whose $\F_p$-cohomology is isomorphic to the exterior Stanley-Reisner $\F_p$-algebra associated to the digraph.
In particular, one has the following conjecture, which is the oriented version of the conjecture formulated in \cite[p.~672]{QSV}.

\begin{conj}
 Let $\Gamma=(\calV,\calE)$ be a special digraph, and let $G$ be the oriented pro-$p$ RAAG associated to $\Gamma$ and to a $p$-power $q$ {\rm(}$q\neq2$ in case $p=2${\rm)}.
 Then $$\bfH^\bullet(G)\simeq\bfLam_\bullet(\Gamma^\ast).$$
\end{conj}

%%%%%%%%%%%%%%%%%%%%%%%%%%%%%%%55

%%%%%%%%%%%%%%%%%%%%%%%%%%%%%%%%%%%%%%%%%%%%%%%%%%%%%%%%%%%%%%%%%%%%%%%%5
%%%%%%%%%
%%%%%%%%%%%%%%%%%%%%%%%%%%%%%%%%%%%%%%%%%%%%%%%%%%%%%%%%%%%%%%%%%%%%%%%%

\section{Massey products}
\label{sec:massey}

\subsection{Massey products in $\F_p$-cohomology}
In the following two subsections, we recall some basic definitions and properties of Massey products in the $\F_p$-cohomology of pro-$p$ groups, and the ``translation'' of Massey products in terms of upper unitriangular representations --- this is all we will need to prove Theorem~\ref{thm:main intro}.
For the formal definition of Massey products in $\F_p$-cohomology of pro-$p$ groups in terms of cochains see, e.g., \cite[\S~2.1]{cq:massey}.

\begin{defin}\rm
Let $G$ be a pro-$p$ group, let $n$ be a positive integer, $n\geq2$, and let $\alpha_1,\ldots,\alpha_n$ be a sequence of elements of $\rmH^1(G)$.
\begin{itemize}
 \item[(a)] The $n$-fold Massey product $\langle\alpha_1,\ldots,\alpha_n\rangle$ is said to be {defined} if it is non-empty.
 \item[(b)] The $n$-fold Massey product $\langle\alpha_1,\ldots,\alpha_n\rangle$ is said to {vanish} if $0\in \langle\alpha_1,\ldots,\alpha_n\rangle$.
 \item[(c)] The $n$-fold Massey product $\langle\alpha_1,\ldots,\alpha_n\rangle$ is said to be {essential} if it is defined but it does not vanish.
 \end{itemize}
 \end{defin}

Massey products satisfy the following (see, e.g., \cite[Rem.~2.2 and Prop.~2.6]{cq:massey}).

\begin{prop}\label{prop:massey}
 Let $G$ be a pro-$p$ group, and let $\alpha_1,\ldots,\alpha_n$ be a sequence of elements of $\rmH^1(G)$.
 Then one has the following.
 \begin{itemize}
  \item[(i)] if the $n$-fold Massey product $\langle\alpha_1,\ldots,\alpha_n\rangle$ is defined then
\begin{equation}\label{eq:cup triv}
 \alpha_1\smallsmile\alpha_2=\alpha_2\smallsmile\alpha_3=\ldots=\alpha_{n-1}\smallsmile\alpha_n=0;
 \end{equation}
 \item[(ii)] if $\alpha_i=0$ for some $i$, then the $n$-fold Massey product $\langle\alpha_1,\ldots,\alpha_n\rangle$ vanishes.
\end{itemize}
\end{prop}

A pro-$p$ group $G$ is said to satisfy the {$n$-Massey vanishing property}, $n\geq2$, with respect to $\F_p$, if every defined $n$-fold Massey product in $\bfH^\bullet(G)$ vanishes.
For this reason, Mina\v c-T\^an's conjecture is also called the ``Massey vanishing conjecture''.
Moreover, $G$ is said to satisfy the {strong} $n$-Massey vanishing property, with respect to $\F_p$, if
every sequence $\langle\alpha_1,\ldots,\alpha_n\rangle$ of length $n$ of elements of $\rmH^1(G)$ satisfying condition \eqref{eq:cup triv} yields a vanishing $n$-fold Massey product (cf. \cite[Def.~1.2]{pal:Massey}).
Clearly, the strong Massey vanishing property is stronger than the Massey vanishing property.
In \cite[Question~1.5]{cq:massey} it is asked whether the maximal pro-$p$ Galois group $G_{\K}(p)$ of a field $\K$ containing a root of 1 of order $p$ has the strong $n$-Massey vanishing property for every $n>2$, if it is a finitely generated pro-$p$ group (see also \cite[Question~4.8]{cq:pyt}).

%%%%%%%%%%%%%%%%%%%%%%%%%%%%%%%%%%%%%%%%%%%%%%%%%%%%%%%%%%%%%%%%%%%%%%%%5
%%%%%%%%%
\subsection{Upper unitriangular matrices}

For $n\geq 1$ let
$$\dbU_{n+1}=\left\{\left(\begin{array}{ccccc} 1 & a_{1,2}&\cdots&& a_{1,n+1} \\ &1&a_{2,3}&\cdots&a_{2,n+1} \\
&&\ddots &\ddots&\vdots \\ &&&1& a_{n,n+1} \\ &&&&1                    \end{array}\right)\mid a_{i,j}\in\F_p\right\}\subseteq\mathrm{GL}_{n+1}(\F_p)$$
be the $p$-group of upper unitriangular matrices with entries in $\F_p$.
The center of $\dbU_{n+1}$ is
\[
 \Zen(\dbU_{n+1})=\{\:I_{n+1}+aE_{1,n+1}\:\mid\:a\in\F_p\:\},
\]
where $I_{n+1}$ denote the $(n+1)\times(n+1)$-identity matrix, and for $1\leq i<j\leq n+1$ let $E_{i,j}$ denote the $(n+1)\times(n+1)$-matrix whose $(i,j)$-entry is 1, and all other entries are 0.
The projection on the $(1,n+1)$-entry yields an isomorphism (of cyclic groups of order $p$) $\Zen(\dbU_{n+1})\simeq\F_p$.
We put $$\bar\dbU_{n+1}=\dbU_{n+1}/\Zen(\dbU_{n+1}).$$

%%%%%%%%%%%%%%%%%%%%%%%%%%%%%%%%%%%%%%%%%%%%%%%%%%%%%%%%%%%%%%%%%%%%%%%%5
%%%%%%%%%
\subsection{Upper unitriangular matrices and Massey products}

Let $G$ be a pro-$p$ group, and let $\rho\colon G\to\dbU_{n+1}$ be a homomorphism of pro-$p$ groups.
For every $i=1,\ldots,n$ the $(i,i+1)$-entry of $\rho$, denoted by $\rho_{i,i+1}$, is a homomorphism $G\to\F_p$, and thus it may be considered as an element of $\rmH^1(G)$.
Analogously, if $n\geq2$ and $\bar\rho\colon G\to\bar\dbU_{n+1}$ is a homomorphism of pro-$p$ groups, then for every $i=1,\ldots,n$ the $(i,i+1)$-entry of $\bar\rho$, denoted by $\bar\rho_{i,i+1}$, is a homomorphism $G\to\F_p$, and thus it may be considered as an element of $\rmH^1(G)$ as well.

The following is the pro-$p$ version of the ``translation'' of Massey products in $\F_p$-cohomology in terms of upper unitriangular representations due to W.~Dwyer (cf., e.g., \cite[Lemma~9.3]{eq:kummer}, see also \cite[\S~8]{ido:Massey}).

\begin{prop}\label{prop:unipotent representation}
Let $G$ be a pro-$p$ group and let $\alpha_1,\ldots,\alpha_n$ be a sequence of elements of $\rmH^1(G)$, with $n\geq2$.
\begin{itemize}
 \item[(i)] The $n$-fold Massey product $\langle\alpha_1,\ldots,\alpha_n\rangle$ is defined if, and only if, there exists a continuous homomorphism $ \bar\rho\colon G\to\bar\dbU_{n+1}$ such that $\bar\rho_{i,i+1}=\alpha_i$ for every $i=1,\ldots,n$.
 \item[(ii)] The $n$-fold Massey product $\langle\alpha_1,\ldots,\alpha_n\rangle$ vanishes if, and only if, there exists a continuous homomorphism $ \rho\colon G\to\dbU_{n+1}$ such that $\rho_{i,i+1}=\alpha_i$ for every $i=1,\ldots,n$.
\end{itemize}
\end{prop}

\subsection{Three lemmata on upper unitriangular matrices}\label{ssec:lemmata}

Here we provide four technical lemmata on upper unitriangular matrices which will be used to prove Theorem~\ref{thm:main intro}.

\begin{lem}\label{lemma:Jordan}
 Let $G$ be a pro-$p$ group, and let $\rho\colon G\to\dbU_{n+1}$ be a homomorphism of pro-$p$ groups for some $n\geq3$.
 Suppose that $\rho_{i,i+1}(x)=1$ for every $i=1,\ldots,n$ for some element $x\in G$.
 Then there exists a homomorphism of (pro-$p$) groups $\rho'\colon G\to\dbU_{n+1}$ such that $\rho'_{i,i+1}(y)=\rho_{i,i+1}(y)$ for all $i=1,\ldots,n$ and $y\in G$, and $ \rho'(x)=A$, where
 \begin{equation}\label{eq:matrix A}
 A=I_{n+1}+E_{1,2}+\ldots+E_{n,n+1}=\left(\begin{array}{ccccc} 1&1&0& &0 \\ &1&1&\ddots& \\ &&\ddots&\ddots&0 \\ &&&1&1 \\ &&&&1 \end{array}\right).
 \end{equation}
\end{lem}

\begin{proof}
Set $A'=\rho(x)$.
 We claim that there exists a basis $\{v_1,\ldots,v_{n+1}\}$ of $\F_p^{n+1}$ such that
 \begin{equation}\label{eq:proof Jordan basis}
 A'v_{i}=v_{i}+v_{i-1} \qquad\text{and}\qquad 
  v_i=\left(\begin{array}{c} b_{i,1}\\ \vdots \\ b_{i,i-1}  \\1  \\0 \\\vdots  \end{array}\right)
 \end{equation}
--- namely, the $i$-th coordinate of $v_i$ is, and the coordinates from the $(i+1)$-th on are 0 --- for every $i=1,\ldots,n+1$.

Clearly, we may set $v_1=(1,0,\ldots)^T$ and $v_2=(0,1,0\ldots)^T$.
Now suppose that $\{v_1,\ldots,v_m\}$ is a set of vectors satisfying \eqref{eq:proof Jordan basis}, for $m\leq n$.
Then the linear system $Aw=w+v_m$, with $w=(x_1,\ldots,x_m,x_{m+1},0,\ldots)^T$, --- explicitly
\[
\setlength\arraycolsep{0.2pt}
\renewcommand\arraystretch{1.25}
\left\{\begin{array}{ccccccccccccc}
 x_{1} &+& x_2 &+& \ldots&\ldots& &+& a_{1,m+1}x_{m+1} &=& x_1 &+& b_{m,1}\\
&& x_2 &+& x_3 &+& \ldots &+& a_{2,m+1}x_{m+1} & =& x_2 &+& b_{m,2}\\
&&&&&&& &\vdots &&&\\
 &&&& x_{m-1} &+& x_m &+& a_{m-1,m+1}x_{m+1} &=& x_{m-1} &+& b_{m,m-1} \\
 &&&&&& x_m &+& x_{m+1} &=& x_m&+&1\\
 &&&&&&&& x_{m+1} &=&x_{m+1}&&
\end{array}\right.\]
--- has a solution with $x_{m}=v_{m+1}=1$, and we may set $v_{m+1}:=w$.

Let $M\in\mathrm{GL}_{n+1}(\F_p)$ be the matrix whose columns are the vectors $v_1,\ldots,v_{n+1}$.
Then $A=M^{-1}A'M$.
Moreover, $M\in \dbU_{n+1}$, and therefore the conjugation of $\dbU_{n+1}$ with $M$ is an inner automorphism of $\dbU_{n+1}$.
Hence, we set $\rho'=M^{-1}\rho M$, so that $\Img(\rho')\subseteq\dbU_{n+1}$, and one has $\rho'(x)=A$.
Finally, the equality $\rho'_{i,i+1}(y)=\rho_{i,i+1}(y)$ is satisfied for all $i=1,\ldots,n$, as for any $B,C\in\dbU_{n+1}$, the $(i,i+1)$-entries of $C^{-1}BC$ and of $B$ are equal.
\end{proof}

\begin{lem}\label{lemma:comm matrices diagonals}
 For $n\geq3$, let $A\in\dbU_{n+1}$ be as in \eqref{eq:matrix A}, and let $B\in\dbU_{n+1}$ be a matrix (with entries $b_{i,j}$ for $1\leq i<j\leq n-1$) such that $[A',B]\in\Zen(\dbU_{n+1})$.
 Then $b_{i,i+k}=b_{1,1+k}$ for every $k=1,\ldots,n-2$ and $i=1,\ldots,n+1-k$ --- namely, 
\begin{equation}\label{eq:matrix B}
B =\left(\begin{array}{ccccccc} 1&b_{1}&b_{2}& \cdots &b_{n-2}&b_{1,n}&b_{1,n+1} \\ &1&b_1&\ddots&&b_{n-2}&b_{2,n+1} \\ 
&&1&\ddots&\ddots&&b_{n-2} \\&&&\ddots&\ddots&\ddots& \vdots \\ &&&&1&b_1&b_{2}  \\ &&&&&1&b_{1} \\ &&&&&&1 \end{array}\right), 
\end{equation}
where $b_k=b_{1,1+k}$ for $k=1,\ldots,n-1$.
\end{lem}

\begin{proof}
 Since $AB\equiv BA\bmod \Zen(\dbU_{n+1})$, for $k=2,\ldots,n-1$ and $i=1,\ldots,n+1-k$ the $(i,i+k)$-entries of the matrices $AB$ and $BA$ are equal --- observe that the the $(i,i+1)$-entries of the two products are always equal, as the $(i,i+1)$-entries of $[A,B]$ are 0.
 Therefore,
 \[
  b_{i,i+k}+b_{i+1,i+k}=b_{i,i+k-1}+b_{i,i+k},
 \]
which implies $b_{i,i+(k-1)}=b_{i+1,i+1+(k-1)}$.
\end{proof}

\begin{lem}\label{lem: B C comm}
For $n\geq3$, let $B,C\in\dbU_{n+1}$ be matrices with $B$ as in \eqref{eq:matrix B}, and 
\begin{equation}\label{eq:matrix C}
C= \left(\begin{array}{cccccc} 1&1&c_{1,3}&\cdots & c_{1,n}&c_{1,n+1} \\ &1&0&\ddots& &c_{2,n+1} \\ &&1&\ddots&\ddots&\vdots
 \\ &&&\ddots&0&c_{n-1,n+1} \\ &&&&1&1 \\ &&&&&1 \end{array}\right),
\end{equation}
and suppose that $[C,B]\in\Zen(\dbU_{n+1})$.
Then $b_1=\ldots=b_{n-2}=0$.
\end{lem}

\begin{proof}
  Since $CB\equiv BC\bmod \Zen(\dbU_{n+1})$, for $k=2,\ldots,n-1$ and the $(1,1+k)$-entries of the matrices $CB$ and $BC$ are equal --- observe that the the $(1,2)$-entries of the two products are always equal, as the $(1,2)$-entry of $[C,B]$ is 0.
 Therefore,
 \begin{equation}\label{eq:equality entries BC CB}
 \begin{split}
  & b_{k}+b_{k-1}+c_{1,3}b_{k-2}+\ldots + c_{1,1+(k-2)}b_2 + c_{1,1+(k-1)}b_1+ c_{1,1+k}= \\
  & =  c_{1,1+k}+b_1 c_{2,1+k}+b_2c_{3,1+k}+\ldots+b_{k-2}c_{k-1,1+k}+b_{k-1}\cdot 0+b_k.
 \end{split}  
 \end{equation}
If $k=2$, then equality \eqref{eq:equality entries BC CB} is $b_2+b_1+c_{1,3}=c_{1,3}+b_2$, which implies that $b_1=0$.
For arbitrary $k$, if $b_1=\ldots=b_{k-2}=0$, then equality \eqref{eq:equality entries BC CB} is 
$b_k+b_{k-1}+c_{1,1+k}=c_{1,1+k}+b_k$, which implies that also $b_{k-1}$ is 0.
\end{proof}

\begin{lem}
For $n\geq3$ and $c\in\F_p$, $a\neq0$, let $C\in\dbU_{n+1}$ be a matrix --- with entries $c_{i,j}$ for $1\leq i<j\leq n+1$ ---, such that $c_{i,i+1}\in\{0,a\}$ for every $i=1,\ldots,n$, and $c_{i,j}=0$ for $j-i\geq2$.
Then for every $f\geq1$ ($f\geq2$ if $p=2$) there exists a matrix $B\in\dbU_{n+1}$ --- with entries $b_{i,j}$ for $1\leq i<j\leq n+1$ --- such that $b_{1,2}=\ldots=b_{n,n+1}=0$ and 
\[
 [B,C]=C^{p^f}.
\]
\end{lem}

\begin{proof}
 Set $q=p^f$.
 First, observe that for every $1\leq i,j\leq n+1$, $j-i\geq2$, one has
 \[
  [I_{n+1}+bE_{i,j},C]=I_{n+1}-\epsilon_1abE_{i-1,j}+\epsilon_2abE_{i,j+1},
 \]
where 
\[
 \epsilon_1=\begin{cases} 1 & \text{if }b_{j-1,j}=a,\\ 0 & \text{if }b_{j-1,j}=0, \end{cases}
 \qquad\text{and}\qquad
 \epsilon_2=\begin{cases} 1 & \text{if }b_{i,i+1}=a,\\ 0 & \text{if }b_{j-1,j}=0. \end{cases}
 \qedhere\]
\end{proof}

%%%%%%%%%%%%%%%%%%%%%%%%%%%%%%%%%%%%%%%%%%%%%%%%%%%%%%%%%%%%%%%%%%%%%%%%5
%%%%%%%%%
%%%%%%%%%%%%%%%%%%%%%%%%%%%%%%%%%%%%%%%%%%%%%%%%%%%%%%%%%%%%%%%%%%%%%%%%
%%%%%%%%%%%%%%%%%%%

\section{Oriented pro-$p$ RAAGs and Massey products}\label{sec:massey RAAGs}

%%%%%%%%%%%%%%%%%%%%%%%%%%%%%%%%%%%%%%%%%%%%%%%%%%%%%%%%%%%%%%%%%%%%%%%%5
%%%%%%%%%
\subsection{Digraphs that are not special-clique}\label{ssec:i}

Recall that a digraph $\Gamma=(\calV,\calE)$ is special if, and only if, it does not contain induced subdigraphs as in \eqref{eq:subdigraphs no special}, and it is also special-clique if, and only if, one excludes also induced subdigraphs as in \eqref{eq:subdigraph no specialclique}.
For our purposes, it is convenient to reformulate the above conditions in the following way: a digraph $\Gamma=(\calV,\calE)$ is not a special-clique special digraph if, and only if, it contains an induced subdigraph $\Gamma'$ with geometric representation
\begin{equation}\label{eq:nospecialclique 1}
  \xymatrix@R=1.5pt{&w& \\& \bullet &  \\
  \circ\ar[ur] & & \circ\ar[ul] \\ u && v}\qquad\text{or}\qquad
  \xymatrix@R=1.5pt{&w& \\& \bullet\ar[dr] & \\
  \circ\ar[ur] & & \circledast\ar@/_0.8pc/@{.>}[ul] \\ u && v}
\end{equation}
(in the right-side representation $(w,v)\in\calE$, and possibly also $(v,w)\in \calE$); or
an induced subdigraph $\Gamma'$ with geometric representation
\begin{equation}\label{eq:nospecialclique 2}
  \xymatrix@R=1.5pt{&w& \\& \bullet\ar[ddr] & \\  \\
  \circledast\ar[uur]\ar@/_/@{<->}[rr] & & \bullet \\ u && v}\qquad\text{or}\qquad
  \xymatrix@R=1.5pt{&w& \\& \bullet\ar@{-}[ddr] &  \\ \\
  \circledast\ar[uur]\ar@/_/@{<->}[rr] & & \circledast \\ u && v}
\end{equation}
where the two-headed arrows between $u$ and $v$ mean that at least one of $(u,v)$ and $(v,u)$ --- possibly both --- are edges.

Our goal is to show that in both cases, the $\F_p$-cohomology of an associated oriented pro-$p$ RAAGs gives rise to essential Massey products.

%%%%%%%%%%%%%%%%%%%%%%%%%%%%%%%%%%%%%%%%%%55

\subsubsection{First case: induced subdigraph of type~\eqref{eq:nospecialclique 1}}

\begin{prop}\label{prop:nomassey nospecialstar1}
Let $\Gamma=(\calV,\calE)$ be a digraph as in \eqref{eq:nospecialclique 1}, and let $G$ be the oriented pro-$p$ RAAG associated to $\Gamma$ and to $q$.
Then the $q$-fold Massey product
\begin{equation}\label{eq:masseyprod nospecialclique case1}
  \langle\alpha,\underbrace{\beta,\ldots,\beta}_{(q-2)\text{times}},\alpha\rangle,
\end{equation}
with $\alpha=u^\ast+v^\ast,\beta=u^\ast\in\rmH^1(G)$, is essential.
\end{prop}

\begin{proof}
First observe that $\alpha\smallsmile\beta=v^\ast\smallsmile u^\ast=0$, as $u,v$ are disjoint, while clearly $\beta\smallsmile\beta=0$.

The vertices $u,v,w$ are subject to the relations $r_{u,w}$ and $r_{v,w}$, which are
 \begin{equation}\label{eq:rel nospecialclique case1}
  [w,u]=u^q\qquad\text{and}\qquad [w,v]=  \begin{cases}
  v^q &\text{if } (v,w)\text{is a directed edge} \\ 1 &\text{if } (v,w),(w,v)\in\calE, \\
   w^{-q} &\text{if } (w,v)\text{is a directed edge}.
                                        \end{cases} \end{equation}
To show that the $q$-fold Massey product \eqref{eq:masseyprod nospecialclique case1} is defined, we produce a homomorphism $\bar\rho\colon G\to \bar\dbU_{q+1}$ satisfying $\bar\rho_{1,2}=\rho_{q,q+1}=\alpha$ and $\bar\rho_{i,i+1}=\beta$ for $i=2,\ldots,q-1$.
Take the matrices $A,C\in\dbU_{q+1}$, with $A$ as in \eqref{eq:matrix A}, and
\[
 C=\left(\begin{array}{ccccc} 1&1&0&\cdots &0\\&1&0&\ddots& \vdots \\&&\ddots&0&0 \\&&&1&1 \\ &&&&1  \end{array}
\right).
\]
Then $A^q=I_{q+1}+E_{1,q+1}$ and $C^q=I_{q+1}$.
The assignment
$$u\mapsto A\cdot \Zen(\dbU_{q+1}),\qquad v\mapsto C\cdot\Zen(\dbU_{q+1}), \qquad w,w'\mapsto I_{q+1}\cdot\Zen(\dbU_{q+1})$$ for any $w'\in\calV$, $w'\neq u,v$,
induces a
homomorphism $\bar\rho\colon G\to \bar\dbU_{q+1}$: indeed one has
\[
 [I_{q+1},A]=I_{q+1}\equiv A^q\bmod \Zen(\dbU_{q+1})\qquad\text{and}\qquad
 [I_{q+1},C]=I_{q+1}=C^q=I_{q+1}^{-q},
\]
as prescribed by \eqref{eq:rel nospecialclique case1}; and moreover
\[
 [\bar\rho(w'),\bar\rho(z)]=I_{q+1}=C^q\equiv A^q\mod \Zen(\dbU_{q+1})
\]
for $z=u,v,w$, so that the relation $r_{w',z}$, occurring whenever a vertex $w'$ is joined to one of $u,v,w$, is satisfied.
Altogether, the $q$-fold Massey product \eqref{eq:masseyprod nospecialclique case1} is defined by Proposition~\ref{prop:unipotent representation}--(i).

Now suppose there exists a homomorphism $\rho\colon G\to \dbU_{q+1}$ satisfying $\rho_{1,2}=\rho_{q,q+1}=\alpha$ and $\rho_{i,i+1}=\beta$ for $i=2,\ldots,q-1$.
By Lemma~\ref{lemma:Jordan}, we may suppose that $\rho(u)=A$, with $A$ as above --- and as in \eqref{eq:matrix A}.
Put $B=\rho(w)$ and $C=\rho(v)$.
Then the entries of the 1st upper diagonal of $B$ are 0, while $C$ is as in Lemma~\ref{lem: B C comm}.
In particular,
\[
 B^q=C^q=I_{q+1}.
\]
By \eqref{eq:rel nospecialclique case1}, one has $[B,A]=A^q\in\Zen(\dbU_{q+1})$, so that Lemma~\ref{lemma:comm matrices diagonals} implies that $B$ as in \eqref{eq:matrix B}.
Moreover, by \eqref{eq:rel nospecialclique case1}, $[B,C]\in\Zen(\dbU_{q+1})$, and thus Lemma~\ref{lem: B C comm} implies that $b_2=\ldots=b_{q-2}=0$ --- namely, the non-0 entries of $B$ are concentrated in the $(q-1)$th and in the $q$th upper diagonals (and in the main diagonal, of course).
Hence, one computes
\[ [B,A]=[B,C]=I_{q+1}+(b_{1,q}-b_{2,q+1})E_{1,q+1},\]
but by the former commutator should be equal to $A^q=I_{q+1}+E_{1,1+q}$, while the latter should be equal to $I_{q+1}$, a contradiction.
Therefore, a homomorphism $\rho\colon G\to\dbU_{q+1}$ with the prescribed properties cannot exist,
and thus the $q$-fold Massey product  \eqref{eq:masseyprod nospecialclique case1} does not vanish by Proposition~\ref{prop:unipotent representation}--(ii).
\end{proof}

\begin{rem}\rm
Let $G$ be the oriented pro-$p$ RAAG as in Proposition~\ref{prop:nomassey nospecialstar1}.
By \cite[Prop.~5.4, Prop.~6.5]{BQW}, $G$ cannot occur as the maximal pro-$p$ Galois group of a field containing a root of 1 of order $p$.
If $p=q=3$, then $G$ yields essential 3-fold Massey products, and thus by E.~Matzri's result \cite{eli:Massey} (cf. \S~\ref{ssec:frame}), Proposition~\ref{prop:nomassey nospecialstar1} provides an alternative proof of the fact that, in this case, $G$ cannot occur as the maximal pro-3 Galois group of a field containing a root of 1 of order 3.
\end{rem}

%%%%%%%%%%%%%%%%%%%%%%%%%%%%%%%%%%%%%%%%%%55

\subsubsection{Second case: induced subdigraph of type~\eqref{eq:nospecialclique 2}}

\begin{prop}\label{prop:nomassey nospecialstar2}
Let $\Gamma=(\calV,\calE)$ be a digraph as in \eqref{eq:nospecialclique 2}, and let $G$ be the oriented pro-$p$ RAAG associated to $\Gamma$ and to $q$.
Then the $q$-fold Massey product
\begin{equation}\label{eq:masseyprod nospecialclique case1}
  \langle\underbrace{\alpha,\ldots,\alpha}_{q\text{times}}\rangle,
\end{equation}
with $\alpha=u^\ast+v^\ast\in\rmH^1(G)$, is essential.
\end{prop}

\begin{proof}
 The vertices $u,v,w$ are subject to the three relations $r_{u,v}$, $r_{v,w}$, $r_{w,u}$, which are
 \begin{equation}\label{eq:rel nospecialclique case2}
 \begin{split}
  [w,u]&=u^q,\\
  [v,w]&=\begin{cases} 1 &\text{if } (v,w),(w,v)\in\calE, \\
  w^q &\text{if } (w,v)\text{is a directed edge}, \end{cases}\\
  [v,u]&=\begin{cases} 1 &\text{if } (v,u),(u,v)\in\calE, \\
  u^q &\text{if } (u,v)\text{is a directed edge},\\
  v^{-q} &\text{if } (v,u)\text{is a directed edge}.\end{cases}
   \end{split}
 \end{equation}
Put $\alpha=u^\ast+v^\ast\in\rmH^1(G)$.
Clearly $\alpha\smallsmile\alpha=0$.
We claim that the $q$-fold Massey product
\begin{equation}\label{eq:masseyprod nospecialclique case2}
  \langle\underbrace{\alpha,\ldots,\alpha}_{q\text{times}}\rangle
\end{equation}
is defined but does not vanish.

To show that the $q$-fold Massey product \eqref{eq:masseyprod nospecialclique case2} is defined, we produce a homomorphism $\bar\rho\colon G\to \bar\dbU_{q+1}$ satisfying $\bar\rho_{i,i+1}=\alpha$ for $i=1,\ldots,q$.
Take the matrix $A\in\dbU_{q+1}$ as in \eqref{eq:matrix A}.
Then $A^q=I_{q+1}+E_{1,q+1}$.
The assignment
$$u,v\mapsto A\cdot \Zen(\dbU_{q+1}), \qquad w,w'\mapsto I_{q+1}\cdot\Zen(\dbU_{q+1})$$ for $w'\in\calV, w'\neq u,v$, induces a homomorphism $\bar\rho\colon G\to \bar\dbU_{q+1}$: indeed one has
\[
 [I_{q+1},A]=I_{q+1}\equiv A^q\bmod \Zen(\dbU_{q+1})\qquad\text{and}\qquad
 [A,A]=I_{q+1}\equiv A^{\pm q}\bmod \Zen(\dbU_{q+1}),
\]
as prescribed by \eqref{eq:rel nospecialclique case2}; and moreover
\[   [\rho(w'),\rho(z)] =I_{q+1}\equiv A^q\mod\Zen(\dbU_{q+1}),\]
  so that the relation $r_{w',z}$, occurring whenever any vertex $w'$ is joined to one of $z=u,v,w$, is satisfied.
Hence, the $q$-fold Massey product \eqref{eq:masseyprod nospecialclique case2} is defined
by Proposition~\ref{prop:unipotent representation}--(i).

Now suppose there exists a homomorphism $\rho\colon G\to \dbU_{q+1}$ satisfying $\rho_{i,i+1}=\alpha$ for $i=1,\ldots,q$.
By Lemma~\ref{lemma:Jordan}, we may suppose that $\rho(u)=A$, with $A$ as above --- and as in \eqref{eq:matrix A}.
Put $B=\rho(w)$ and $C=\rho(v)$.
Then the entries of the 1st upper diagonal of $B$ are 0, while all entries of the 1st upper diagonal of $C$ are equal to 1.
In particular,
\[
 B^q=I_{q+1}\qquad\text{and}\qquad C^q=A^q=I_{q+1}+E_{1,q+1}.
\]
By \eqref{eq:rel nospecialclique case2}, one has $[B,A],[A,C]\in\Zen(\dbU_{q+1})$, so that Lemma~\ref{lemma:comm matrices diagonals} implies that $B,C$ are as in \eqref{eq:matrix B} (where we call $c_i$, $i=2,\ldots,q-2$, the entries of the $i$th upper diagonal of $C$).
Then one computes
\[
 [B,A]=I_{q+1}+(b_{1,q}-b_{2,q+1})E_{1,q+1},
\]
and $[B,A]=A^q=I_{q+1}+E_{1,q+1}$ implies $b_{1,q}-b_{2,q+1}=1$.
On the other hand, the $(1,q+1)$-entry of $[C,B]$ is
\[\begin{split}
 &\left(b_{1,q+1}+b_{2,q+1}+b_{q-2}c_2+\ldots+b_2c_{q-2}+0+c_{1,q+1}\right)-\\
&\qquad  -\left(c_{1,1+q}+0+b_2c_{q-2}+\ldots+b_{q-2}c_2+b_{1,q}+b_{1,q+1}\right)=b_{2,q+1}-b_{1,q},
\end{split}\]
and $[C,B]=B^{\epsilon q}=I_{q+1}$ (with $\epsilon=0,1$ depending on whether $(v,w)$ is an edge) implies $b_{2,q+1}-b_{1,q}=0$, a contradiction.
Therefore, a homomorphism $\rho\colon G\to\dbU_{q+1}$ with the prescribed properties cannot exist,
and thus the $q$-fold Massey product  \eqref{eq:masseyprod nospecialclique case2} does not vanish by Proposition~\ref{prop:unipotent representation}--(ii).
\end{proof}

%%%%%%%%%%%%%%%%%%%%%%%%%%%%%%%%%%%%%%%%%%%%%%%%%%%%%%%%%%%%%%%%%%%%%%%%5
%%%%%%%%%
\subsection{Special-clique digraphs and Massey products}\label{ssec:ii}

Here we prove that condition~(iii) in Theorem~\ref{thm:main intro} implies condition~(ii).
We will proceed as follows: given a sequence $\alpha_1,\ldots,\alpha_n$ of elements of $\rmH^1(G,\F_p)$ satisfying \eqref{eq:cup triv}, we will construct explicitly a homomorphism $\rho\colon G\to\dbU_{n+1}$ such that $\rho_{i,i+1}=\alpha_i$ for every $i=1,\ldots,n$, so that the $n$-fold Massey product $\langle\alpha_1,\ldots,\alpha_n\rangle$ contains 0 by Proposition~\ref{prop:special clique inclusion}--(ii).
At this aim, we need the following lemma, which is a consequence of \cite[Prop.~2.10]{cq:massey}.

\begin{lem}\label{lem:comm qpower}
 For $a\in\F_p$, $a\neq0$, let $A\in\dbU_{n+1}$ be a matrix whose non-0 entries are concentrated in the main diagonal and in the 1st upper diagonal --- namely,
 \[
  A=I_{n+1}+a_{1,2}E_{1,2}+a_{2,3}E_{2,3}+\ldots+a_{n,n+1}E_{n,n+1}.
 \]
Then there exists a matrix $B\in\dbU_{n+1}$ whose entries in the 1st upper diagonal are 0 and such that $[B,A]=A^q$.
\end{lem}

\begin{proof}
First, we write $A$ as a block diagonal matrix,
\[
 A=\left(\begin{array}{cccc} A_1 &&& \\ &A_2&& \\ &&\ddots& \\ &&&A_r\end{array}
\right)
\]
where for each $h=1,\ldots,r$ one has $A_h\in\dbU_{m_h+1}$ for some $m_h\geq0$, and either $A_h=I_{m_h+1}$, or the entries of the 1st upper diagonal of $A_h$ are all not 0.
For each block, one has either $A_h^q=I_{m_h+1}$, if $A_h=I_{m_h+1}$ or $m_h<q$; or the non-0 entries of $A_h^q$ are concentrated in the main diagonal and in the $q$th upper diagonal --- namely,
\[
 A_h^q=I_{m_h+1}+a(E_{1,1+q}+\ldots+E_{m_h+1-q,m_h+1})=
 \left(\begin{array}{ccccc}  1&\cdots&\ast&& 0\\ &1&0&\ddots & \\ &&\ddots&0&\ast \\ &&&1&\vdots \\ &&&&1 \end{array}
\right),
\]
if the entries of the 1st upper diagonal of $A_h$ are all not 0 and $m_h\geq q$.
Altogether, one has
\[
 A^q=\left(\begin{array}{cccc} A_1^q &&& \\ &A_2^q&& \\ &&\ddots& \\ &&&A_r^q\end{array}
\right).
\]

Now, if $A_h$ is a block whose entries in the 1st upper diagonal of $A_h$ are not 0, and $m_h\geq q$, by \cite[Prop.~2.10]{cq:massey} there exists a matrix $B_h\in\dbU_{m_h+1}$, whose entries in the 1st diagonal are 0, such that $[B_h,A_h]=A_h^q$.
If instead $A_h^q=I_{m_h+1}$, we put $B_h=I_{m_h+1}$, so that $[B_h,A_h]=I_{m_h+1}=A_h^q$ anyway.

Finally, let $B\in\dbU_{n+1}$ be the diagonal block matrix whose blocks are $B_1,\ldots,B_r$.
Then
\[
\begin{split}
 [B,A] &=\left(\begin{array}{cccc} [B_1,A_1] &&& \\ &[B_2,A_2]&& \\ &&\ddots& \\ &&&[B_r,A_r]\end{array}\right) \\
&= \left(\begin{array}{cccc} A_1^q &&& \\ &A_2^q&& \\ &&\ddots& \\ &&&A_r^q\end{array}
\right)=A^q.\qedhere
\end{split}\]
\end{proof}

\begin{prop}\label{thm:massey specialclique}
 Let $\Gamma=(\calV,\calE)$ be a special-clique special digraph, and let $G$ be the oriented pro-$p$ RAAG associated to $\Gamma$ and to $q$.
 Then $G$ satisfies the strong $n$-Massey vanishing property for every $n\geq3$.
\end{prop}

\begin{proof}
If $\Gamma$ has no special vertices, then $\Gamma$ is an undigraph, and the associated pro-$p$ RAAG satisfies the  $n$-Massey vanishing property for every $n\geq3$ by \cite[Thm.~1.1]{BCQ}.
So, we suppose that $\Gamma$ has at leas a special vertex.

 Let $\alpha_1,\ldots,\alpha_n$ be a sequence of elements of $\rmH^1(G,\F_p)$ satisfying \eqref{eq:cup triv}.
 For every vertex $v\in\calV$, set
 \[\begin{split}
  A(v)&=I_{n+1}+\alpha_1(v)E_{1,2}+\alpha_2(v)E_{2,3}+\ldots+\alpha_n(v)E_{n,n+1}\\
    &=\left(\begin{array}{ccccc}  1&\alpha_1(v)&0&& 0\\ &1&\alpha_2(v)&\ddots & \\ &&\ddots&&0 \\ &&&1&\alpha_n(v) \\ &&&&1 \end{array}\right)\in\dbU_{n+1} .  \end{split}\]
By the proof of \cite[Thm.~1.1]{BCQ}, one has
\begin{equation}\label{eq:ab-ba}
  \alpha_i(v)\alpha_{i+1}(v')-\alpha_i(v')\alpha_{i+1}(v)=0\qquad\text{for every }i=1,\ldots,n-1,
\end{equation}
whenever $v$ and $v'$ are adjacent ordinary vertices, which implies that
\begin{equation}\label{eq:rel1 specialclique}
[A(v),A(v')]=I_{n+1}\qquad\text{for every }v,v'\in\calV_0
\end{equation}
 (cf. \cite[p.~13]{BCQ}).

Now let $w\in\calV_s$ be a special vertex of $\Gamma$, and put $\mathrm{St}(w)=(\mathcal{W},\calE_w)$.
Since $\Gamma$ is special-clique, $\mathrm{St}(w)$ is a clique of $\Gamma$, and by Proposition~\ref{prop:special clique inclusion} the associated oriented pro-$p$ RAAG $G_w$ is a subgroup of $G$ via the inclusion $\mathcal{W}\hookrightarrow\calV$
Moreover, the $\F_p$-cohomology algebra of $G_w$ is the exterior Stanley-Reisner $\F_p$-algebra
$$\bfLam_\bullet(\mathrm{St}(w))=\bfLam_\bullet(\F_p\mathcal{W}^\ast)=\bfLam_\bullet(\rmH^1(G_w))$$
(for example, by Theorem~\ref{thm:quad intro}).
For every $i=1,\ldots,n-1$, one has
\begin{equation}\label{eq:cup res in star}
 0=\res_{G,G_w}^2(\alpha_i\smallsmile\alpha_{i+1})=(\alpha_i\vert_{G_w})\smallsmile(\alpha_{i+1}\vert_{G_w}),
\end{equation}
and since $\rmH^2(G_w)=\Lambda_2(\rmH^1(G_2))$, \eqref{eq:cup res in star} implies that $\alpha_{i}\vert_{G_w}$ and $\alpha_{i+1}\vert_{G_w}$ are $\F_p$-linearly dependent.
Therefore, there exists $\bar\alpha\in\rmH^1(G_w)$ and $a_1,\ldots,a_n\in\F_p$ such that
$$\alpha_{i}\vert_{G_w}=a_i\bar\alpha\qquad \text{for all }i=1,\ldots,n.$$
Here we have two cases.

\medskip

\noindent {Case 1.}
Suppose that $\bar\alpha(u)=0$ for every ordinary vertex $u$ of $\mathrm{St}(w)$.
Then $A(u)=I_{n+1}$, and thus $[A(v),A(u)]=I_{n+1}$ for every other ordinary vertex $v$ of $\Gamma$, and
$$[A(w'),A(u)]=I_{n+1}=A(u)^q$$ for every special vertex $w'$ of $\Gamma$.
\medskip

\noindent {Case 2.}
Suppose now that $\bar\alpha(u)\neq 0$ for some ordinary vertex $u$ of $\mathrm{St}(w)$.
Then for every vertex $u'\in\mathcal{W}$ (including $u'=u,w$) one has
\begin{equation}\label{eq:alfa u alfa uprimo}
\alpha_i(u')=a_i\bar\alpha(u')=\frac{\bar\alpha(u')}{\bar\alpha(u)}\cdot a_i\bar\alpha(u)=
\frac{\bar\alpha(u')}{\bar\alpha(u)}\cdot\alpha_i(u)\qquad\text{for all }i=1,\ldots,n.
\end{equation}
For every $u'\in\mathcal{W}$, $u'\neq w$, replace the matrix $A(u')$ with $A'(u')=A(u)^{\bar\alpha(u')/\bar\alpha(u)}$.
Observe that for every $i=1,\ldots,n$ the $(i,i+1)$-entry of the latter matrix is precisely $\frac{\bar\alpha(u')}{\bar\alpha(u)}\bar\alpha(u)$, which is equal to $\alpha_i(u')$ by \eqref{eq:alfa u alfa uprimo}.
In particular, if $\bar\alpha(u')=0$ then $A'(u')=I_{n+1}=A(u')$, while clearly $A'(u)=A(u)$.

If $v\in\calV_0$ is an ordinary vertex of $\Gamma$ adjacent to an ordinary vertex $u'$ of $\mathrm{St}(w)$, then by \eqref{eq:ab-ba} and by \eqref{eq:alfa u alfa uprimo} for every $i=1,\ldots,n$ one has
\[ \begin{split}
  0&= \alpha_i(v)\alpha_{i+1}(u')-\alpha_i(u')\alpha_{i+1}(v)\\
   &= \alpha_i(v)\frac{\bar\alpha(u')}{\bar\alpha(u)}\alpha_{i+1}(u)-\frac{\bar\alpha(u')}{\bar\alpha(u)}\alpha_i(u)\alpha_{i+1}(v)\\
   &= \frac{\bar\alpha(u')}{\bar\alpha(u)}\cdot\left(\alpha_i(v)\alpha_{i+1}(u)-\alpha_i(u)\alpha_{i+1}(v)\right)
 \end{split}\]
Hence, if $\bar\alpha(u')\neq0$ then $\alpha_i(v)\alpha_{i+1}(u')=\alpha_i(u')\alpha_{i+1}(v)$, and \eqref{eq:rel1 specialclique} implies the equality $[A(v),A(u)]=I_{n+1}$, and thus also
$$[A(v),A'(u')]=\left[A(v),A(u)^{\bar\alpha(u')/\bar\alpha(u)}\right]=I_{n+1}.$$
On the other hand, if $\bar\alpha(u')=0$, then $A'(u')=I_{n+1}$, and $[A(v),A'(u')]=I_{n+1}$ trivially.

Finally, let $B\in\dbU_{n+1}$ be a matrix whose entries in the 1st upper-diagonal are 0 and such that $[B,A(u)]=A(u)^q$ --- cf. Lemma~\ref{lem:comm qpower} ---, and set $A'(w)=BA(u)^{\bar\alpha(w)/\bar\alpha(u)}$.
Then for every $i=1,\ldots,n$ the $(i,i+1)$-th entry of $A'(w)$ is $\frac{\bar\alpha(w)}{\bar\alpha(u)}\bar\alpha(u)$, which is equal to $\alpha_i(w)$; moreover, for every $u'\in\mathcal{W}$, $u'\neq w$ (but possibly $u'=u$) one computes, applying \eqref{eq:comm times}--\eqref{eq:comm power},
\[
\begin{split}
 [A'(w),A'(u')] &= \left[BA(u)^{\bar\alpha(w)/\bar\alpha(u)},A(u)^{\bar\alpha(u')/\bar\alpha(u)}\right]\\
 & =  \left[B,A(u)^{\bar\alpha(u')/\bar\alpha(u)}\right] \\
 & = [B,A(u)]^{\bar\alpha(u')/\bar\alpha(u)}=\left(A(u)^{\bar\alpha(u')/\bar\alpha(u)}\right)^q=A(u')^q.
\end{split}\]

\medskip

\noindent
Altogether, the assignment $v\mapsto A(v)$ or $v\mapsto A'(v)$ induces a homomorphism $\rho\colon G\to \dbU_{n+1}$ satisfying $\rho_{i,i+1}=\alpha_i$ for all $i=1,\ldots,n$, and thus the $n$-fold Massey product $\langle\alpha_1,\ldots,\alpha_n\rangle$ vanishes by Proposition~\ref{prop:unipotent representation}--(ii).
\end{proof}

\begin{rem}\rm
 In the case $\Gamma=(\calV,\calE)$ is an undigraph (which is trivially special-clique), from Proposition~\ref{thm:massey specialclique} one recovers \cite[Thm.~1.1]{BCQ}, which states that the pro-$p$ RAAG associated to an undigraph satisfies the strong $n$-Massey vanishing property for every $n>2$.
\end{rem}

 %%%%%%%%%%%%%%%%%%%%%%%%%%%%%%%%%%%

 %%%%%%%%%%%%%%%%%%%%%%%%%%%%%%%%%%%
Now we are ready to prove Theorem~\ref{thm:main intro}, by putting together Propositions~\ref{prop:nomassey nospecialstar1}--\ref{prop:nomassey nospecialstar2} and Proposition~\ref{thm:massey specialclique}.

\begin{proof}[Proof of Theorem~1.1]
In order to prove the implication (i)$\Rightarrow$(iii), suppose that $\Gamma=(\calV,\calE)$ is not a special-clique special digraph, and let $G$ be the oriented pro-$p$ RAAG associated to $\Gamma$ and to $q$.
Then there exists an induced subdigraph $\Gamma'=(\calV',\calE')$ with three vertices, which is as in \eqref{eq:nospecialclique 1} or in \eqref{eq:nospecialclique 2}, and Propositions~\ref{prop:nomassey nospecialstar1}--\ref{prop:nomassey nospecialstar2} imply that there exist essential $q$-fold Massey products in $\bfH^\bullet(G)$. This completes the proof of (i)$\Rightarrow$(iii).

The implication (iii)$\Rightarrow$(ii) is provided by Proposition~\ref{thm:massey specialclique}.

Finally, the implication (ii)$\Rightarrow$(i) follows by definition.
\end{proof}

 %%%%%%%%%%%%%%%%%%%%%%%%%%%%%%%%%%%

 %%%%%%%%%%%%%%%%%%%%%%%%%%%%%%%%%%%

 %%%%%%%%%%%%%%%%%%%%%%%%%%%%%%%%%%%

 %%%%%%%%%%%%%%%%%%%%%%%%%%%%%%%%%%%

\end{document}